\documentclass[12pt,a4paper]{article}
\usepackage{geometry}
\usepackage{hyperref}

\usepackage[font=small,labelfont=bf]{caption}

\usepackage{times,authblk}
\usepackage[UKenglish]{babel}
\renewcommand{\phi}{\varphi}
\renewcommand{\rho}{\varrho}
\renewcommand{\epsilon}{\varepsilon}
\usepackage[cal=esstix]{mathalfa}  
\usepackage{amsmath, amsfonts, amssymb, amsthm, bm, stmaryrd, enumerate, mathtools}
\usepackage{graphicx}
\usepackage{tikz}
\usetikzlibrary{positioning, decorations.markings}
\usepackage{tikz-cd}
\tikzset{commutative diagrams/.cd}

\newtheorem{Def}{Definition}[section]
\newenvironment{definition}{\begin{Def} \rm}{\end{Def}}
\newtheorem{lemma}[Def]{Lemma}
\newtheorem{proposition}[Def]{Proposition}

\newtheorem{theorem}[Def]{Theorem}
\newtheorem{example}[Def]{Example}
\newtheorem{remark}[Def]{Remark}

\newcommand{\komma}{,\hspace{0.3em}}
\newcommand{\id}{\text{\rm id}}
\renewcommand{\leq}{\leqslant}
\renewcommand{\geq}{\geqslant}

\newcommand{\Reals}{{\mathbb R}}
\newcommand{\Complexes}{{\mathbb C}}

\newcommand{\notperp}{\mathbin{\not\perp}}
\newcommand{\perpe}[1]{\mathbin{\perp_{#1}}}

\newcommand{\Zero}{{\bf 0}}
\renewcommand{\c}{^\perp}
\newcommand{\cc}{^{\perp\perp}}

\newcommand{\ce}[1]{^{\perp_{#1}}}
\newcommand{\cce}[1]{^{\perp_{#1}\perp_{#1}}}

\newcommand{\lin}[1]{\langle #1\rangle}
\renewcommand{\P}{\mathbf P}
\newcommand{\prl}{\parallel}
\newcommand{\spec}{\preccurlyeq}

\newcommand{\class}[1]{[ #1 ]}
\newcommand{\zerokernel}[1]{{#1}^{\circ}}

\newcommand{\withoutzero}{^{\raisebox{0.2ex}{\scalebox{0.4}{$\bullet$}}}}
\newcommand{\adj}{^\star}
\newcommand{\adjadj}{^{\star\star}}

\newcommand{\C}{{\mathcal C}}
\newcommand{\cOL}{\mathcal{cOL}}
\newcommand{\cOML}{\mathcal{cOML}}
\newcommand{\caOL}{\mathcal{caOL}}
\newcommand{\OS}{\mathcal{OS}}
\newcommand{\supOS}{^\text{\rm OS}}
\newcommand{\iOS}{\mathcal{iOS}}
\newcommand{\FOS}{\mathcal{FOS}}
\newcommand{\iDS}{\mathcal{iDS}}
\newcommand{\Hil}[1]{\mathcal{Hil}_{#1}}
\newcommand{\PHil}[1]{\mathcal{PHil}_{#1}}
\newcommand{\op}{^\text{op}}
\newcommand{\principalideal}[1]{{#1\hspace{-3pt}\downarrow}}
\DeclareMathOperator{\kernel}{ker}
\DeclareMathOperator{\image}{im}
\DeclareMathOperator{\homset}{hom}
\DeclareMathOperator{\card}{card}
\newcommand{\One}{{\mathbf 1}}

\begin{document}

\title{Categories of orthosets and adjointable maps}

\author[1]{Jan Paseka}

\author[2]{Thomas Vetterlein}

\affil[1]{\footnotesize Department of Mathematics and Statistics,
Masaryk University \authorcr
Kotl\'a\v rsk\'a 2, 611\,37 Brno, Czech Republic \authorcr
{\tt paseka@math.muni.cz}}

\affil[2]{\footnotesize Institute for Mathematical Methods in Medicine and Data Based Modeling, \authorcr
Johannes Kepler University Linz \authorcr
Altenberger Stra\ss{}e 69, 4040 Linz, Austria \authorcr
{\tt Thomas.Vetterlein@jku.at}}

\date{\today}

\maketitle

\begin{abstract}\parindent0pt\parskip1ex

\noindent An orthoset is a non-empty set together with a symmetric and irreflexive binary relation $\perp$, called the orthogonality relation. An orthoset with $0$ is an orthoset augmented with an additional element $0$, called falsity, which is orthogonal to every element. The collection of subspaces of a Hilbert space that are spanned by a single vector provides a motivating example.

We say that a map $f \colon X \to Y$ between orthosets with $0$ possesses the adjoint $g \colon Y \to X$ if, for any $x \in X$ and $y \in Y$, $f(x) \perp y$ if and only if $x \perp g(y)$. We call $f$ in this case adjointable. For instance, any bounded linear map between Hilbert spaces induces a map with this property. We discuss in this paper adjointability from several perspectives and we put a particular focus on maps preserving the orthogonality relation.

We moreover investigate the category $\OS$ of all orthosets with $0$ and adjointable maps between them. We especially focus on the full subcategory $\iOS$ of irredundant orthosets with~$0$. $\iOS$ can be made into a dagger category, the dagger of a morphism being its unique adjoint. $\iOS$ contains dagger subcategories of various sorts and provides in particular a framework for the investigation of Hilbert spaces.

{\it Keywords:} Orthoset; orthogonality space; Hermitian space; Hilbert space; dagger category

{\it MSC:} 81P10; 06C15; 46C05

\mbox{}\vspace{-2ex}

\end{abstract}

\section{Introduction}
\label{sec:Introduction}

Orthogonality is a concept omnipresent in mathematics. To explain its significance is nevertheless not an easy issue. For sure, however, we can say that numerous structures commonly used in linear algebra, geometry, or mathematical physics can be built exclusively on this very notion. Particularly remarkably, a Hilbert space, serving as the basic model of quantum physics, can in a certain sense be reduced to its orthogonality relation.

An {\it orthoset}, also called an {\it orthogonality space}, is a non-empty set $X$ equipped with a symmetric, irreflexive binary relation $\perp$, referred to as an orthogonality relation \cite{Dac,Mac,Hav,Wlc,DiNg}. The notion originates from the logico-algebraic approach to the foundation of quantum mechanics. It was once coined by David Foulis and his collaborators, the guiding example being the collection of one-dimensional subspaces of a Hilbert space together with the usual orthogonality relation \cite{Dac}. Indeed, this simple structure is of considerable significance: the orthoset arising from a Hilbert space $H$ leads directly to the ortholattice of its closed subspaces, which in turn is known to allow the reconstruction of $H$.

In spite of the conceptual simplicity, it is not straightforward to decide how to organise orthosets into a category. Orthosets can be identified with undirected graphs and in this context several possibilities have been investigated, see, e.g., \cite{Wal,Pfa}. It seems natural to require a morphism to preserve orthogonality \cite{PaVe1,PaVe2}. It has turned out, however, that this idea is of limited use when the focus is on connections with inner-product spaces. The present work reconsiders the issue and is based on another idea: we suppose morphisms to possess adjoints.

To begin with, we use in this paper a slightly adjusted version of the main notion. An {\it orthoset with $0$} is defined similarly to an orthoset, but comes with an additional element $0$ that is orthogonal to all elements. On the one hand, this harmless modification helps to avoid technical complications. On the other hand, our guiding example is essentially the same as before: rather than considering the collection of one-dimensional subspaces, we consider the collection of subspaces spanned by single vectors. Subsequently, we shall refer to orthosets with $0$ simply as ``orthosets''.

Let $f \colon X \to Y$ be a map between orthosets (with $0$). We call a further map $g \colon Y \to X$ an {\it adjoint} of $f$ if, for any $x \in X$ and $y \in Y$, $f(x) \perp g$ if and only if $x \perp g(y)$. To express that $f$ possesses an adjoint, we say that $f$ is {\it adjointable}. We may say that adjointable maps generalise orthogonality-preserving ones. For, a bijective map $f \colon X \to Y$ has the adjoint $f^{-1}$ if and only if $f$ preserves and reflects the orthogonality relation. In the context of inner-product spaces, however, adjointability comes closer to linearity than to unitarity. Indeed, any linear map between finite-dimensional Hilbert spaces induces an adjointable map between the associated orthosets. In fact, so does any bounded linear map between arbitrary Hilbert spaces.

We investigate in this paper, first, the basic facts around adjointable maps. For instance, viewing orthosets as closure spaces, we see that adjointable maps are continuous. A particular focus is moreover on maps that preserve, in a possibly restricted sense, the orthogonality relation. For instance, our framework is well suited to deal with partial orthometries, a notion defined by analogy with partial isometries between Hilbert spaces. Finally, we observe that we are naturally led to Dacey spaces in the present context. Namely, assume that $X$ is an orthoset such that the inclusion map of any of its subspaces to $X$ is adjointable. Then ${\mathsf C}(X)$, the ortholattice of subspaces of $X$, is orthomodular, which means that $X$ is Dacey. If, in addition, $X$ is atomistic, ${\mathsf C}(X)$ is an atomistic lattice with the covering property.

We turn, second, to the issue of finding a suitable category of orthosets. Needless to say, we consider categories whose objects are orthosets and whose morphisms are adjointable maps between them. We start by considering the category $\OS$ of all orthosets and the adjointable maps between them. We characterise the monos and epis in $\OS$ and point out that equalisers exist only in particular cases. Furthermore, we consider the category $\iOS$ of irredundant orthosets and adjointable maps. An orthoset is irredundant if two elements have the same orthocomplement only in case when they are equal. As adjoints are in this case unique, $\iOS$ is actually a dagger category \cite{AbCo,Sel}. We note that Jacobs studied in \cite{Jac} a category of orthomodular lattices and it turns out that our approach is closely related to his; a detailed discussion can be found in \cite{BPL}. Here, we show that there is a faithful and dagger essentially surjective dagger-preserving functor from $\iOS$ to the dagger category of complete ortholattices, which restricts to a functor from the dagger category of irredundant Dacey spaces to the dagger category of complete orthomodular lattices.

A follow-up paper will be devoted to an issue that naturally arises in the present context. Indeed, an obvious question is how a category consisting of Hilbert spaces, or more general Hermitian spaces, can be described in the present framework by putting suitable conditions on $\OS$.

Our paper is structured as follows. In the subsequent Section~\ref{sec:Orthosets-with-0}, we provide basic definitions and information on orthosets, and especially on Dacey spaces. Section~\ref{sec:Adjointable-maps} is devoted to the concept of adjointability of maps between orthosets. Section~\ref{sec:Orthometric-correspondences} explains in which sense the concept of adjointability can be used to describe orthogonality-preserving maps. In the final two sections, we turn to category theory. We discuss in Sections~\ref{sec:OS} and~\ref{sec:iOS} the categories $\OS$ and $\iOS$ of all orthosets and of the irredundant orthosets, respectively.

\section{Orthosets with $\mathbf 0$}
\label{sec:Orthosets-with-0}

The basic model used in quantum physics is the complex Hilbert space. Following the suggestion of D.\ Foulis, we may discard all its structure and keep solely its orthogonality relation, to be led to a minimalist model. An {\it orthoset}, also known as an {\it orthogonality space}, is based on a binary relation that is assumed to be symmetric and irreflexive \cite{Dac}. There is no reference to any field, linear operation, topology, or whatever. For an introduction to orthosets, we refer the reader to \cite{Wlc}. The question which properties characterise the orthosets of one-dimensional subspaces of inner-product spaces is discussed, e.g., in \cite{Vet1,Vet2,Rum,PaVe3}.

The present section contains an introduction to the key aspects of orthosets. We recall that orthosets have the natural structure of a closure space \cite{Ern}. In particular, the collection of orthoclosed subspaces forms a complete ortholattice. Moreover, orthosets may be distinguished with regard to separation properties, such as being irredundant or Fr\' echet. Finally, the orthosets arising from Hilbert spaces are so-called Dacey spaces, which we briefly discuss as well.

Our subsequent definition of an orthoset differs from Foulis's original one. Namely, we consider orthosets that are augmented with an additional element denoted by $0$, which is supposed to be orthogonal to all elements.

\begin{definition}
An {\it orthoset with $0$} is a non-empty set $X$ equipped with a binary relation $\perp$ called the {\it orthogonality relation} and with a constant $0$ called {\it falsity}. Moreover, the following is assumed:
\begin{itemize}

\item[\rm (O1)] $x \perp y$ implies $y \perp x$ for any $x, y \in X$,

\item[\rm (O2)] $x \perp x$ if and only if $x = 0$, 

\item[\rm (O3)] $0 \perp x$ for any $x \in X$.

\end{itemize}
An element of $X$ distinct from falsity is called {\it proper}; we put $X\withoutzero = X \setminus \{0\}$.

Elements $x, y \in X$ such that $x \perp y$ are called {\it orthogonal}. By a {\it $\perp$-set}, we mean a subset of an orthoset $X$ consisting of mutually orthogonal proper elements. The supremum of the cardinalities of $\perp$-sets in $X$ is called the {\it rank} of $X$.
\end{definition}

To simplify matters, we will drop in the sequel the attribute ``with $0$''. Throughout this paper, an orthoset will generally be meant to be an orthoset with $0$.

Informally, an orthoset might be thought of encoding maximal consistent collections of properties of a physical system. Orthogonality then corresponds to mutual exclusion and falsity stands for contradiction. More specifically, what we have in mind is the collection of pure states of a quantum-mechanical system, together with a further entity encoding impossibility. Our guiding example is the following.

\begin{example} \label{ex:basic-1}
Let $H$ be a real or complex Hilbert space. Then $H$, equipped with the usual orthogonality relation and with the zero vector as falsity, is an orthoset. Note that the (Hilbert) dimension of $H$ is the cardinality of any maximal $\perp$-set and hence coincides with the rank of $H$.

We get a modified version of this example by switching to the projective structure. This is what we actually consider as the guiding example of an orthoset. Let us denote the subspace spanned by a vector $u \in H$ by $\lin u$ and define
\[ \P(H) \;=\; \{ \lin u \colon u \in H \}. \]
That is, $\P(H)$ is the set of all subspaces of $H$ spanned by a single vector, including the zero vector. Defining the orthogonality relation again in the usual way and choosing the zero linear subspace $\lin 0 = \{0\}$ as falsity, we make $\P(H)$ into an orthoset. Obviously, the rank of \/ $\P(H)$ is equal to the rank, and hence the dimension, of $H$.

In what follows, we shall deal with an orthoset derived from $H$ slightly differently. For $u \in H$, let $\class u = \{ \alpha \, u \colon \alpha \neq 0 \}$ and define
\begin{equation} \label{fml:PH}
P(H) \;=\; \{ \class u \colon u \in H \}.
\end{equation}
As before we equip $P(H)$ with the usual orthogonality relation and the constant $\class 0 = \{0\}$. We observe that $P(H)$ arises from $\P(H)$ simply by removing the $0$ vector from all non-zero subspaces. Thus $P(H)$ can be identified with $\P(H)$ in an obvious way.
\end{example}

The notion of an orthoset leads directly to the realm of lattice theory and we will shortly mention the basic facts. Recall that an {\it orthoposet} is a bounded poset that is additionally equipped with an order-reversing involution $\c$ sending each element $a$ to a complement of $a$. For two elements $a$, $b$ of an orthoposet, we put $a \perp b$ if $a \leq b\c$. Moreover, an {\it ortholattice} is an orthoposet that is lattice-ordered.

The {\it orthocomplement} of a subset $A$ of an orthoset $X$ is
\[ A\c \;=\; \{ x \in X \colon x \perp y \text{ for all } y \in A \}. \]
The subsets $A$ of $X$ such that $A = A\cc$ are called {\it orthoclosed} and we denote by ${\mathsf C}(X)$ the collection of all orthoclosed subsets of $X$. Partially ordered by set-theoretic inclusion, ${\mathsf C}(X)$ is a complete lattice, $\{0\}$ being the smallest and $X$ the largest element. Additionally equipped with the orthocomplementation, ${\mathsf C}(X)$ becomes an ortholattice.

The restriction of the orthogonality relation to any subset $A$ of $X$ containing $0$ leads likewise to an orthoset, which we call a {\it suborthoset} of $X$. When dealing the same time with $A$ and $X$, we continue denoting the orthocomplementation on $X$ by $\c$, whereas we write $\ce{A}$ for the orthocomplementation on the suborthoset $A$. That is, we put $B\ce{A} = B\c \cap A$ for $B \subseteq A$.

Suborthosets arise in particular as components of decompositions, cf.\ \cite{PaVe2}. For $k \geq 1$, a {\it $k$-ary decomposition} of $X$ is a collection $(A_1, \ldots, A_k)$ of subsets of $X$ such that $A_i = \big(\bigcup_{j \neq i} A_j\big)\c$ for each $i$. Note that a subset $A$ of $X$ is orthoclosed if and only if $A$ is the constituent of some decomposition. In fact, $A$ is orthoclosed if and only if there is a further subset $B$ of $X$ such that $(A,B)$ is a binary decomposition. An orthoclosed set always contains the falsity element and is hence a suborthoset, which we call a {\it subspace} of $X$. In particular, we refer to $\{0\}$ as the {\it zero subspace} of $X$.

A contrasting situation is described in the following lemma. Under certain circumstances, the ortholattices associated with an orthoset and its suborthoset can be identified.

\begin{lemma} \label{lem:isomorphism-of-ortholattice-of-suborthoset}
Let $Y$ be a suborthoset of the orthoset $X$. Assume that, for any $x \in X\withoutzero$, there is a subset $A \subseteq Y$ such that $\{x\}\c = A\c$. Then the maps
\begin{align*}
& {\mathsf C}(X) \to {\mathsf C}(Y) \komma A \mapsto A \cap Y, \\
& {\mathsf C}(Y) \to {\mathsf C}(X) \komma A \mapsto A\cc
\end{align*}
are mutually inverse isomorphisms.
\end{lemma}

\begin{proof}
Note that $(\{x\}\cc \cap Y)\cc = \{x\}\cc$ for any $x \in X$. For $A \in {\mathsf C}(X)$, we hence have $A = \bigvee_{x \in A} \{x\}\cc = \bigvee_{x \in A} (\{x\}\cc \cap Y)\cc = \big(\bigcup_{x \in A} (\{x\}\cc \cap Y)\big)\cc = (A \cap Y)\cc$ and $A = (A\c \cap Y)\c$. In particular, $A \cap Y = (A\c \cap Y)\ce{Y} \in {\mathsf C}(Y)$. Moreover, for $A \in {\mathsf C}(Y)$, we have that $ A =  A\cce{Y} = (A\c \cap Y)\c \cap Y = A\cc \cap Y$. Hence the indicated maps are mutually inverse bijections. Clearly, both are order-preserving and we conclude that they establish an isomorphism of lattices. As $(A \cap Y)\ce{Y} = (A \cap Y)\c \cap Y = A\c \cap Y$ for any $A \in {\mathsf C}(X)$, this is an isomorphism of ortholattices.
\end{proof}

Orthosets give rise to ortholattices. We see next that, conversely, orthoposets lead to orthosets.

For the Dedekind-MacNeille completion of orthoposets, see, e.g., \cite{Mac,Pal}. For a subset $A$ of an orthoposet $L$ to be {\it join-dense}, we mean that any element of $L$ is the join of some (not necessarily finite) subset of $A$. For an element $a$ of $L$, we write $\principalideal a = \{ x \in L \colon x \leq a \}$.

\begin{proposition} \label{prop:orthosets-and-ortholattices}
Let $L$ be an orthoposet and let $X$ be a join-dense subset of $L$ containing the bottom element $0$. Then $X$, equipped with the orthogonality relation inherited from $L$ and with $0$, is an orthoset. Moreover, ${\mathsf C}(X)$ together with the injection
\[ \iota_L \colon L \to {\mathsf C}(X) \komma a \mapsto \{ x \in X \colon x \leq a \} \]
is the Dedekind-MacNeille completion of $L$. Finally,
\[ \iota_X \colon X \to {\mathsf C}(X) \komma x \mapsto \{x\}\cc \]
is an injection such that, for any $x, y \in X$, $x \perp y$ iff $\iota_X(x) \perp \iota_X(y)$, and $\iota_X(0) = \{0\}$.
\end{proposition}

\begin{proof}
Equipped with $\perp$ and $0$, $L$ is evidently an orthoset. For $A \subseteq L$, let $A^\uparrow$ be the set of upper bounds of $A$ in $L$ and $A^\downarrow$ the set of lower bounds of $A$ in $L$. We have
\[ \begin{split} A^{\uparrow\downarrow}
& \;=\; \{ x \in L \colon x \leq y \text{ for any $y \in L$ such that } y \geq z \text{ for any } z \in A \} \\
& \;=\; \{ x \in L \colon x \perp y \text{ for any $y \in L$ such that } y \perp z \text{ for any } z \in A \}
\;=\; A\cc. \end{split} \]
For $a \in L$, we moreover have $\principalideal a = \{a\}\cc \in {\mathsf C}(L)$. We conclude that ${\mathsf C}(L)$, together with the injection $L \to {\mathsf C}(L) \komma a \mapsto \principalideal a$, is the Dedekind-MacNeille completion of the orthoposet $L$.

Let now $X$ be a join-dense subset of $L$. Equipped with $\perp$ and $0$, $X$ is a suborthoset of $L$. By Lemma~\ref{lem:isomorphism-of-ortholattice-of-suborthoset}, the map ${\mathsf C}(L) \to {\mathsf C}(X) \komma A \mapsto A \cap X$ is an isomorphism of ortholattices. Hence also ${\mathsf C}(X)$, together with the injection $\iota_L \colon L \to {\mathsf C}(X) \komma a \mapsto \principalideal a \cap X$, is the Dedekind-MacNeille completion of $L$.

Note finally that, for $x \in X$, we have by Lemma~\ref{lem:isomorphism-of-ortholattice-of-suborthoset} that $\iota_X(x) = \{x\}\cce{X} = (\{x\}\c \cap X)\c \cap X = \{x\}\cc \cap X = \principalideal x \cap X = \{ y \in X \colon y \leq x \}$. Hence $\iota_X$ is injective. The remaining assertions about $\iota_X$ are obvious.
\end{proof}

\begin{remark} \label{rem:orthosets-and-ortholattices}
Let $L$ be an orthoposet. Then we may view $L$ according to Proposition~\ref{prop:orthosets-and-ortholattices} as an orthoset. To avoid confusion, we will denote the latter occasionally by $L\supOS$.

The MacNeille completion of $L$ can then be identified with ${\mathsf C}(L\supOS)$. In the particular case that $L$ is complete, we have that $L$ itself can be identified with ${\mathsf C}(L\supOS)$. In this sense, we may say that a complete ortholattice can be reduced to its orthogonality relation.
\end{remark}

For any orthoset $X$, the double orthocomplementation $\cc$ is a closure operator on $X$. That is, $X$, equipped with $\cc$, is a closure space \cite[Section~2.1]{Ern} and the orthoclosed sets are precisely the subsets that are closed w.r.t.\ $\cc$. This point of view will be useful at some places in the sequel.

Separation properties have been considered for closure spaces, generalising the well-known notions for topological spaces \cite{Ern}. In our setting, the following conditions are relevant.

\begin{definition} \label{def:separation-properties}
Let $X$ be an orthoset.
\begin{itemize}

\item[\rm (i)] We call $X$ {\it irredundant} if, for any distinct proper elements $x, y \in X$, there is a $z \in X$ such that either $z \perp x$ but $z \notperp y$, or $z \perp y$ but $z \notperp x$.

\item[\rm (ii)] We call $X$ {\it atomistic} if, for any proper elements $x, y \in X$, the following holds: If there is a $z \in X$ such that $z \perp x$ but $z \notperp y$, then there is also an $z' \in X$ such that $z' \perp y$ but $z' \notperp x$.

\item[\rm (iii)] We call $X$ {\it Fr\' echet} if, for any distinct proper elements $x, y \in X$, there is a $z \in X$ such that $z \perp x$ but $z \notperp y$.

\end{itemize}
\end{definition}

An orthoset $X$, viewed as a closure space, carries the so-called {\it specialisation order} \cite{Ern}, which we denote by $\spec$. That is, for $x, y \in X$ we put $x \spec y$ if $\{x\}\cc \subseteq \{y\}\cc$. Evidently, $\spec$ is a preorder. Moreover, let us call two elements $x, y \in X$ {\it equivalent} if $\{x\}\c = \{y\}\c$; we write $x \prl y$ in this case. Then $x \prl y$ if and only if $x \spec y$ and $y \spec x$. Clearly, $\prl$ is an equivalence relation and we denote the equivalence class of $x \in X$ by $\class x$.

\begin{lemma} \label{lem:irredundant}
For an orthoset $X$, the following are equivalent.
\begin{itemize}

\item[\rm (a)] $X$ is irredundant.

\item[\rm (b)] For any $x, y \in X\withoutzero$, $\{x\}\c = \{y\}\c$ implies $x = y$.

\item[\rm (c)] $\spec$ is antisymmetric, that is, a partial order.

\item[\rm (d)] $\class x = \{x\}$ for any $x \in X$, that is, $\prl$ is equality.

\end{itemize}
\end{lemma}

\begin{proof}
Straightforward.
\end{proof}

In view of criterion (b) of Lemma~\ref{lem:irredundant}, we observe that the irredundancy of an orthoset corresponds to property of a closure space to be $T_0$; see, e.g., \cite{Ste}. 

With an orthoset $X$, we may associate in a canonical way an orthoset that is  irredundant and such that the associated ortholattices can be identified. We call
\[ P(X) \;=\; \{ \class x \colon x \in X \} \]
the {\it irredundant quotient} of $X$. Note, moreover, that for any $A \in {\mathsf C}(X)$ and $x \in A$ we have $\class x \subseteq A$. Hence it makes sense to define $P_X(A) = \{ \class x \colon x \in A \}$.

\begin{proposition} \label{prop:irredundant-quotient}
Let $X$ be an orthoset. On $P(X)$, we may define
\[ \class x \perp \class y \;\text{ if }\; x \perp y, \]
where $x, y \in X$. Endowed with $\perp$ and the constant $\class 0 = \{0\}$, $P(X)$ becomes an irredundant orthoset. We moreover have:
\begin{itemize}

\item[\rm (i)] The map ${\mathsf C}(X) \to {\mathsf C}(P(X)) \komma A \mapsto P_X(A)$ is an isomorphism of ortholattices.

\item[\rm (ii)] $(A_1, \ldots, A_k)$ is a decomposition of $X$ if and only if $(P_X(A_1), \ldots, P_X(A_k))$ is a decomposition of $P(X)$.

\end{itemize}
\end{proposition}

\begin{proof}
Clearly, $P(X)$ can be made into an orthoset in the indicated way. Moreover, for any $x, y \in X$, we have $x \perp y$ in $X$ if and only if $\class x \perp \class y$ in $P(X)$. Hence $\{\class x\}\c = \{\class z\}\c$ in $P(X)$ implies $\{x\}\c = \{z\}\c$ in $X$, that is, $\class x = \class z$. This shows that $P(X)$ is irredundant.

The assertions (i) and (ii) are easily checked.
\end{proof}

In reverse perspective, we may say that any orthoset arises from an irredundant orthoset by a ``multiplication'' of its elements; cf.~\cite[Section 2]{Vet2}. It might hence seem that to describe $P(X)$ is essentially the same task as to describe $X$. However, this is true only when $X$ is considered in isolation; in a categorical context, the ``multiplicity'' of elements may well play an important role.

\begin{lemma} \label{lem:atomistic}
For an orthoset $X$, the following are equivalent:
\begin{itemize}

\item[\rm (a)] $X$ is atomistic.

\item[\rm (b)] For any $x, y \in X\withoutzero$, $\{x\}\c \subseteq \{y\}\c$ implies $\{x\}\c = \{y\}\c$.

\item[\rm (c)] $\spec$ coincides with $\prl$.

\item[\rm (d)] $\{x\}\cc = \class x \cup \{0\}$ for any $x \in X\withoutzero$.

\item[\rm (e)] For any $x \in X\withoutzero$, $\{x\}\cc$ is an atom of ${\mathsf C}(X)$.

\item[\rm (f)] $P(X)$ is Fr\' echet.

\end{itemize}
\end{lemma}

\begin{proof}
The equivalence of (a), (b), and (c) is obvious.

(b) $\Rightarrow$ (d): Assume that (b) holds. Clearly, $\class x \cup \{0\} \subseteq \{x\}\cc $ for any $x \in X\withoutzero$. Moreover, $y \in \{x\}\cc \setminus \{0\}$ implies $\{y\}\cc \subseteq \{x\}\cc$ and hence $\{x\}\c = \{y\}\c$, that is, $y \in \class x$. We conclude that $\class x \cup \{0\} = \{x\}\cc$.

(d) $\Rightarrow$ (e): Assume that (d) holds. For any $x, y \in X\withoutzero$ such that $y \in \{x\}\cc$, we have $\{y\}\cc \subseteq \{x\}\cc$ and hence $\class y \subseteq \class x$, that is, $x \prl y$ and $\{y\}\cc = \{x\}\cc$. We conclude that $\{x\}\cc$ is an atom of ${\mathsf C}(X)$.

(e) $\Rightarrow$ (b): Assume that (e) holds. Then, for any $x, y \in X\withoutzero$, $\{x\}\c \subseteq \{y\}\c$ implies $\{0\} \neq \{y\}\cc \subseteq \{x\}\cc$ and hence  $\{x\}\c = \{y\}\c$.

(b) $\Leftrightarrow$ (f): For $x, y \in X\withoutzero$, we have that $\{\class x\}\c \subseteq \{\class y\}\c$ in $P(X)$ iff $\{x\}\c \subseteq \{y\}\c$ in $X$, and $\{\class x\}\c = \{\class y\}\c$ in $P(X)$ iff $\{x\}\c = \{y\}\c$ in $X$. We conclude that $X$ fulfils condition (b) if and only if so does $P(X)$. Moreover, $P(X)$ is by Proposition~\ref{prop:irredundant-quotient} in this case irredundant and hence Fr\' echet.
\end{proof}

As seen next, the atomisticity of an orthoset implies the equally denoted property of the associated ortholattice. Recall that a lattice with $0$ is called {\it atomistic} if the set of atoms is join-dense, that is, if every element is a join of atoms.

\begin{lemma} \label{lem:atomistic-orthoset-atomistic-ortholattice}
Let $X$ be an orthoset. If $X$ is atomistic, then so is ${\mathsf C}(X)$. In fact, ${\mathsf C}(X)$ is atomistic if and only if $X$ possesses an atomistic suborthoset $Y$ such that, for any $x \in X\withoutzero$, there is a subset $A \subseteq Y$ such that $\{x\}\c = A\c$. In this case, the atoms of ${\mathsf C}(X)$ are exactly the sets $\{y\}\cc$, $y \in Y\withoutzero$, and we have:
\begin{itemize}

\item[\rm (i)] The maps ${\mathsf C}(X) \to {\mathsf C}(Y) \komma A \mapsto A \cap Y$ and ${\mathsf C}(Y) \to {\mathsf C}(X) \komma A \mapsto A\cc$ are mutual inverse isomorphisms.

\item[\rm (ii)] $(A_1, \ldots, A_k)$ is a decomposition of $X$ if and only if $(A_1 \cap Y, \ldots, A_k \cap Y)$ is a decomposition of $Y$.

\end{itemize}
\end{lemma}

\begin{proof}
Assume that ${\mathsf C}(X)$ is atomistic. Then each atom is of the form $\{y\}\cc$ for some $y \in X\withoutzero$. Let $Y$ be the set of all $y \in X$ such that $y=0$ or else $\{y\}\cc$ is an atom of ${\mathsf C}(X)$. For any $x \in X\withoutzero$, we then have $\{x\}\cc = \bigvee \big\{ \{y\}\cc \colon y \in Y \text{ such that } \{y\}\cc \subseteq \{x\}\cc \big\} = (\{x\}\cc \cap Y)\cc$, that is, $\{x\}\c = (\{x\}\cc \cap Y)\c$. Let $x, y \in Y\withoutzero$ be such that $\{x\}\ce{Y} \subseteq \{y\}\ce{Y}$. Then, in view of Lemma~\ref{lem:isomorphism-of-ortholattice-of-suborthoset}, we have $\{x\}\c = (\{x\}\c \cap Y)\cc = (\{x\}\ce{Y})\cc \subseteq (\{y\}\ce{Y})\cc = (\{y\}\c \cap Y)\cc = \{y\}\c$. By the atomisticity of ${\mathsf C}(X)$, it follows $\{x\}\c = \{y\}\c$ and hence $\{x\}\ce{Y} = \{y\}\ce{Y}$. We have shown that $Y$ is atomistic.

Conversely, assume that $Y$ is an atomistic suborthoset of $X$ such that $\{x\}\cc = (\{x\}\cc \cap Y)\cc$ for any $x \in X\withoutzero$. Note that then for any $x \in X\withoutzero$ there is a $y \in Y\withoutzero$ such that $\{y\}\cc \subseteq \{x\}\cc$. We claim that, for any $y \in Y\withoutzero$, $\{y\}\cc$ is an atom of ${\mathsf C}(X)$. Let $x \in X\withoutzero$ be such that $\{x\}\cc \subseteq \{y\}\cc$. Choose a $z \in Y\withoutzero$ such that $\{z\}\cc \subseteq \{x\}\cc$. Then $\{y\}\ce{Y} = \{y\}\c \cap Y \subseteq \{z\}\c \cap Y = \{z\}\ce{Y}$ and hence $\{y\}\ce{Y} = \{z\}\ce{Y}$. By Lemma~\ref{lem:isomorphism-of-ortholattice-of-suborthoset}, we conclude $\{y\}\c = (\{y\}\c \cap Y)\cc = (\{y\}\ce{Y})\cc = (\{z\}\ce{Y})\cc = (\{z\}\c \cap Y)\cc = \{z\}\c$ and hence $\{x\}\cc = \{y\}\cc$. The assertion follows and it is then also clear that ${\mathsf C}(X)$ is atomistic. Finally, if $x \in X\withoutzero$ is such that $\{x\}\cc$ is an atom, there is a $y \in Y\withoutzero$ such that $\{x\}\cc = \{y\}\cc$. Hence each atom of ${\mathcal C}(X)$ is of the form $\{y\}\cc$ for some $y \in Y\withoutzero$.

Finally, (i) holds by Lemma~\ref{lem:isomorphism-of-ortholattice-of-suborthoset}, and (ii) follows from (i).
\end{proof}

The third property among those introduced in Definition~\ref{def:separation-properties} is equivalent to the conjunction of the other two.

\begin{lemma} \label{lem:Frechet}
For an orthoset $X$, the following are equivalent.
\begin{itemize}

\item[\rm (a)] $X$ is Fr\' echet.

\item[\rm (b)] For any $x, y \in X\withoutzero$, $\{x\}\c \subseteq \{y\}\c$ implies $x=y$.

\item[\rm (c)] $X$ is irredundant and atomistic.

\item[\rm (d)] $\spec$ is equality.

\item[\rm (e)] For any $x \in X$, $\{x,0\}$ is orthoclosed.

\end{itemize}
\end{lemma}

\begin{proof}
Straightforward.
\end{proof}

In view of condition (e) of Lemma~\ref{lem:Frechet}, we may say that the property of an orthoset to be Fr\' echet corresponds to the property of a closure space to be $T_1$; see, e.g.,~\cite{Ern}.

For Fr\' echet orthosets, we get a one-to-one correspondence between orthosets and their associated ortholattices. An element $p$ of an ortholattice $L$ is called {\it basic} if $p$ is either an atom or the bottom element. We denote by ${\mathsf B}(L)$ the collection of all basic elements of $L$. Equipped with the orthogonality relation and the bottom element of~$L$, $\,{\mathsf B}(L)$ is an orthoset.

An {\it orthoisomorphism} is a bijection $f \colon X \to Y$ between orthosets such that $f(0) = 0$ and, for any $x, y \in X$, $x \perp y$ iff $f(x) \perp f(y)$.

\begin{proposition} \label{prop:Frechet-orthosets}
Let $X$ be a Fr\' echet orthoset. Then ${\mathsf C}(X)$ is a complete atomistic ortholattice and $X \to {\mathsf B}({\mathsf C}(X)) \komma x \mapsto \{x,0\}$ is an orthoisomorphism.

Conversely, let $L$ be a complete atomistic ortholattice. Then ${\mathsf B}(L)$ is a Fr\' echet orthoset and $L \to {\mathsf C}({\mathsf B}(L)) \komma a \mapsto \{ p \in {\mathsf B}(L) \colon p \leq a \}$ is an ortholattice isomorphism.
\end{proposition}

\begin{proof}
For the first part, we recall that, by Lemma~\ref{lem:Frechet}, if the orthoset $X$ is Fr\' echet then $\{x,0\}$ is orthoclosed for any $x \in X$. The second part is clear from Proposition~\ref{prop:orthosets-and-ortholattices}.
\end{proof}

\begin{example} \label{ex:basic-2}
Let $H$ be a Hilbert space, viewed as an orthoset as in Example~\ref{ex:basic-1}. Then $H$ is atomistic. Indeed, for any distinct non-zero vectors $u, v \in H$, the subspaces $\{ x \in H \colon x \perp u \}$ and $\{ x \in H \colon x \perp v \}$ either coincide or are incomparable. Moreover, $H$ is not irredundant, and its irredundant quotient is $P(H)$ as specified in Example~\ref{ex:basic-1}. Note that $P(H)$ is Fr\' echet.
\end{example}

We finally mention the situation that a pair of complementary subspaces exhausts an orthoset. An orthoset $X$ is called {\it reducible} if there is a decomposition $(A,B)$ of $X$ into non-zero subspaces such that $A \cup B = X$; otherwise, we say that $X$ is {\it irreducible}.

Similarly, we call an ortholattice $L$ {\it irreducible} if $L$ is directly indecomposable.

\begin{lemma} \label{lem:irreducible-orthoset}
An atomistic orthoset $X$ is irreducible if and only if ${\mathsf C}(X)$ is irreducible.
\end{lemma}

\begin{proof}
Let $X$ possess a decomposition $(A,A\c)$ such that $A \cup A\c = X$. Then ${\mathsf C}(X)$ is isomorphic to ${\mathsf C}(A) \times {\mathsf C}(A\c)$. The ``if'' part follows.

To see the ``only if'' part, let $\tau \colon {\mathsf C}(X) \to L_1 \times L_2$ be an isomorphism, where $L_1$ and $L_2$ are ortholattices with at least two elements. By Lemma~\ref{lem:atomistic}, for any $x \in X\withoutzero$, $\{x\}\cc$ is an atom of ${\mathsf C}(X)$, hence either $\tau(\{x\}\cc) = (p,0)$, where $p$ is an atom of $L_1$, or $\tau(\{x\}\cc) = (0,q)$, where $q$ is an atom of $L_2$. Let $A = \{ x \in X \colon \tau(\{x\}\cc) \in L_1 \times \{0\} \}$ and $B = \{ x \in X \colon \tau(\{x\}\cc) \in \{0\} \times L_2 \}$. Then $(A,B)$ is a decomposition of $X$ into non-zero subspaces such that $A \cup B = X$.
\end{proof}

\subsection*{Dacey spaces}

The orthosets occurring in the present context are characterised by an additional property that cannot be expressed by a first-order statement. It concerns the situation that we like to consider the subspaces of an orthoset that is itself contained as a subspace in a larger orthoset.

Let us note first that the subspace relation is in the following sense transitive: if $A$ is a subspace of an orthoset $X$ and $B$ is in turn a subspace of $A$, then $B$ is a subspace of $X$. Indeed, in this case $B = B\cce{A} = (B\cc \vee A\c) \cap A \in {\mathsf C}(X)$. However, the subspaces of $X$ that are contained in $A$ are not necessarily subspaces of~$A$.

We call $X$ a {\it Dacey space} if ${\mathsf C}(X)$ is orthomodular. The next lemma shows that the indicated unintuitive situation does not occur exactly if $X$ is a Dacey space.  A further, convenient characterisation of Dacey spaces is Dacey's criterion \cite{Dac,Wlc}, which we include as criterion (d) in the lemma.

\begin{lemma} \label{lem:Dacey-space}
Let $X$ be an orthoset. Then the following are equivalent:
\begin{itemize}

\item[\rm (a)] $X$ is Dacey.

\item[\rm (b)] For any subspace $A$ of $X$ and any $B \subseteq A$, we have $B\cce{A} = B\cc$.

\item[\rm (c)] For any subspace $A$ of $X$, ${\mathsf C}(A) = \{ B \in {\mathsf C}(X) \colon B \subseteq A \}$.

\item[\rm (d)] For any $A \in {\mathsf C}(X)$ and any maximal $\perp$-set $D$ contained in $A$, we have that $A = D\cc$.

\end{itemize}
\end{lemma}

\begin{proof}
(a) $\Rightarrow$ (b): Assume that $X$ is Dacey and $A \in {\mathsf C}(X)$. Then, for any $B \subseteq A$, we have by orthomodularity $B\cce{A} = (B\c \cap A)\c \cap A = (B\cc \vee A\c) \cap A = B\cc$.

(b) $\Rightarrow$ (c): Assume (b). Then $B \subseteq A$ is orthoclosed in the subspace $A$ if and only if $B$ is orthoclosed in $X$.

(c) $\Rightarrow$ (d): Let $A \in {\mathsf C}(X)$ and let $D \subseteq A$ be a maximal $\perp$-set. Then $D\cc \in {\mathsf C}(X)$ and $D\cc \subseteq A$. Assuming (c), we have $D\cc \in {\mathsf C}(A)$. But then $D\cc = (D\cc)\cce{A} = D\cce{A} = A$ because of the maximality of $D$.

(d) $\Rightarrow$ (a): Let $A, B \in {\mathsf C}(X)$ such that $A \subseteq B$. Extend a maximal $\perp$-set $D \subseteq A$ to a maximal $\perp$-set $E \subseteq B$. Assume that (d) holds. Then $(E \setminus D)\cc \perp D\cc = A$, hence $(E \setminus D)\cc \subseteq B \cap A\c$. We conclude $B = E\cc = D\cc \vee (E \setminus D)\cc \subseteq A \vee (B \cap A\c) \subseteq B$. That is, $B = A \vee (B \cap A\c)$, which shows the orthomodularity of ${\mathsf C}(X)$.
\end{proof}

The next lemma compiles some properties that are preserved from Dacey spaces to their subspaces.

We recall that a lattice is said to have the {\it covering property} if, for any element $a$ and atom $p \nleq a$, $\,a \vee p$ covers $a$. By an {\it AC} lattice, we mean an atomistic lattice with the covering property.

\begin{lemma} \label{lem:subspaces-of-Dacey-spaces}
Let $A$ be a subspace of a Dacey space $X$.
\begin{itemize}

\item[\rm (i)] $A$ is Dacey as well.

\item[\rm (ii)] The subspaces of $A$ are the subspaces of $X$ contained in $A$.

\item[\rm (iii)] Let $x, y \in A$. Then $x$ and $y$ are equivalent elements of $A$ if and only if $x$ and $y$ are equivalent elements of $X$. That is, we have $P_X(A) = P(A)$.

\item[\rm (iv)] If $X$ is irredundant, so is $A$.

\item[\rm (v)] If $X$ is atomistic, so is $A$ and the atoms of $A$ are $\{x\}\cc$, $x \in A$.

\item[\rm (vi)] If ${\mathsf C}(X)$ has the covering property, so has ${\mathsf C}(A)$.

\end{itemize}
\end{lemma}

\begin{proof}
Ad (i): For any $B \in {\mathsf C}(A)$ and $C \subseteq B$, we have by criterion (c) of Lemma~\ref{lem:Dacey-space} that $B \in {\mathsf C}(X)$ and hence $C\cce{B} = C\cc = C\cce{A}$ by criterion (b). Hence also $A$ is Dacey by criterion (b).

Ad (ii): This is a reformulation of criterion (c) of Lemma~\ref{lem:Dacey-space}.

Ad (iii): From $\{x\}\c = \{y\}\c$ it follows $\{x\}\ce{A} = \{x\}\c \cap A = \{y\}\c \cap A = \{y\}\ce{A}$. To see the converse, note that, by orthomodularity,
$\{x\}\ce{A} \vee A\c = (\{x\}\c \cap A) \vee A\c = \{x\}\c$ and similarly $\{y\}\ce{A} \vee A\c = \{y\}\c$. Hence $\{x\}\ce{A} = \{y\}\ce{A}$ implies $\{x\}\c = \{y\}\c$.

Ad (iv): This is clear from part (iii).

Ad (v): Let $X$ be atomistic. For any $x \in A$, we have $\{x\}\cce{A} = \{x\}\cc = \class x \cup \{0\}$ by Lemma~\ref{lem:Dacey-space} and Lemma~\ref{lem:atomistic}. Again by Lemma~\ref{lem:atomistic} and by part (iii), it follows that also $A$ is atomistic.

Ad (vi): By criterion (c) of Lemma~\ref{lem:Dacey-space}, the lattice ${\mathsf C}(A)$ is a principal ideal of the lattice ${\mathsf C}(X)$.
\end{proof}

\section{Adjointable maps}
\label{sec:Adjointable-maps}

We now turn to the core issue of the present work: we introduce the concept of adjointability of maps between orthosets. Adjoints of maps between orthosets can be seen as a generalisation of adjoint operators between Hilbert spaces.

We investigate in this section the consequences resulting from the adjointability of mappings. We establish that these maps are compatible with the equivalence of elements and thus induce maps between the irredundant quotients. We observe that adjointable maps, seen as maps between closure spaces, are continuous. Furthermore, we point out that any adjointable map induces a map between the associated ortholattices which is likewise adjointable. Finally, in analogy to linear maps between vector spaces, we relate kernels and images of adjointable maps to injectivity and surjectivity. We also discuss reducing subspaces and we define scalar maps, for which each subspace is invariant.

\begin{definition} \label{def:adjoint}
Let $f \colon X \to Y$ be a map between orthosets. We call $g \colon Y \to X$ an {\it adjoint} of $f$ if, for any $x \in X$ and $y \in Y$,
\[ f(x) \perp y \quad\text{if and only if}\quad x \perp g(y). \]
Moreover, a map $f \colon X \to X$ is called {\it self-adjoint} if $f$ is an adjoint of itself.
\end{definition}

It is clear that adjointness is a symmetric property: if a map $f$ possesses an adjoint $g$, then $f$ is also an adjoint of $g$. To stress the symmetry, we may speak of $f$ and $g$ as an adjoint pair.

Obviously, the identity map $\id \colon X \to X$ is self-adjoint. Moreover, if $g$ is an adjoint of $f \colon X \to Y$ and $k$ is an adjoint of $h \colon Y \to Z$, then obviously $g \circ k$ is an adjoint of $h \circ f$.

In what follows, the question of the existence of an adjoint will be most important. We will call a map $f \colon X \to Y$ between orthosets {\it adjointable} if there is an adjoint $g \colon Y \to X$ of $f$. $f$ is not in general adjointable and if so, the adjoint of $f$ need not be unique.

\begin{example} \label{ex:Adjoints-in-Hilbert-spaces}
Let $H_1$ and $H_2$ be Hilbert spaces and let $\phi \colon H_1 \to H_2$ be a bounded linear map. Viewing $H_1$ and $H_2$ as orthosets, we then have that $\phi$ is adjointable. Indeed, $\phi\adj$, the usual adjoint of $\phi$, is an adjoint of $\phi$ in the sense of Definition~\ref{def:adjoint}. Any non-zero multiple of $\phi\adj$ is likewise an adjoint of $\phi$.

Moreover, $\phi$ induces a map between the irredundant quotients $P(H_1)$ and $P(H_2)$, namely
\begin{equation} \label{fml:Adjoints-in-Hilbert-spaces}
P(\phi) \colon P(H_1) \to P(H_2) \komma \class x \mapsto \class{\phi(x)}.
\end{equation}
This map is adjointable as well: obviously, $P(\phi\adj)$ is an adjoint of $P(\phi)$.
\end{example}

Our first observation is that adjointable maps preserve the equivalence of elements of orthosets. Consequently, they are compatible with the formation of irredundant quotients.

Let $f \colon X \to Y$ be a map between orthosets and assume that $f$ preserves $\prl$. Generalising (\ref{fml:Adjoints-in-Hilbert-spaces}) in Example \ref{ex:Adjoints-in-Hilbert-spaces}, we put
\[ P(f) \colon P(X) \to P(Y) \komma \class x \mapsto \class{f(x)}. \] 

\begin{proposition} \label{prop:quotient-map}
A map $f \colon X \to Y$ between orthosets is adjointable if and only if $f$ preserves $\prl$ and $P(f)$ is adjointable.

In this case, $g \colon Y \to X$ is an adjoint of $f$ if and only if $g$ preserves $\prl$ and $P(g)$ is an adjoint of $P(f)$.
\end{proposition}

\begin{proof}
Let $f$ possess the adjoint $g$. Let $x, x' \in X$ be such that $x \prl x'$. Then, for any $y \in Y$, we have $f(x) \perp y$ iff $x \perp g(y)$ iff $x' \perp g(y)$ iff $f(x') \perp y$. Hence $f(x) \prl f(x')$. We conclude that $f$ preserves $\prl$. Similarly, we see that so does $g$. Moreover, for any $x \in X$ and $y \in Y$, we have that $P(f)(\class x) \perp \class y$ iff $f(x) \perp y$ iff $x \perp g(y)$ iff $\class x \perp P(g)(\class y)$. Hence $P(g)$ is an adjoint of $P(f)$.

Conversely, assume that $f$ preserves $\prl$ and $G \colon P(Y) \to P(X)$ is an adjoint of $P(f)$. Let $g \colon Y \to X$ be any map such that $g(y) \in G(\class y)$ for any $y \in Y$. Then $g$ preserves $\prl$ and $G = P(g)$. The fact that $P(g)$ is an adjoint of $P(f)$ in turn implies that $g$ is an adjoint of $f$.
\end{proof}

With regard to Proposition \ref{prop:quotient-map} we note that, to infer the adjointability of a map $f$ from the adjointability of $P(f)$, the axiom of choice is needed.

Let us call adjointable maps $f, f' \colon X \to Y$ {\it equivalent} if $f(x) \prl f'(x)$ for all $x \in X$. We write $f \prl f'$ in this case. In other words, for $f$ and $f'$ to be equivalent means that $P(f) = P(f')$.

It is immediate that the adjoints of the same map are mutually equivalent.

\begin{lemma} \label{lem:irredundancy-adjoints-1}
Let $X$ and $Y$ be orthosets.
\begin{itemize}

\item[\rm (i)] Let $f \colon X \to Y$ and $g \colon Y \to X$ be an adjoint pair. Then a further map $h \colon Y \to X$ is an adjoint of $f$ if and only if $h \prl g$.

\item[\rm (ii)] $X$ is irredundant if and only if any map from $X$ to $Y$ possesses at most one adjoint.

\end{itemize}
\end{lemma}

\begin{proof}
Ad (i): Let $h \colon Y \to X$ be an adjoint of $f$. Then, for any $y \in Y$, we have $\{g(y)\}\c = \{ x \in X \colon f(x) \perp y \} = \{h(y)\}\c$, that is $g(y) \prl h(y)$. This shows the ``only if'' part; the ``if'' part is obvious.

Ad (ii): By part (i), the irredundancy of $X$ implies the uniqueness of adjoints. For the converse direction, assume that $X$ is not irredundant. Let $b$ and $b'$ be distinct elements of $X$ such that $b \prl b'$. Let moreover $c \in Y$ and
\[ f \colon X \to Y \komma x \mapsto
\begin{cases} c & \text{if $x \notperp b$,} \\ 0 & \text{otherwise.} \end{cases} \]
Then both
\[ g \colon Y \to X \komma y \mapsto
\begin{cases} b & \text{if $y \notperp c$,} \\ 0 & \text{otherwise} \end{cases}
\quad\text{and}\quad
g' \colon Y \to X \komma y \mapsto
\begin{cases} b' & \text{if $y \notperp c$,} \\ 0 & \text{otherwise} \end{cases} \]
are adjoints of $f$.
\end{proof}

We next show that any adjointable map is continuous. In this case, we regard the involved orthosets as closure spaces. Continuity means that the membership of an element in the closure of some set is preserved \cite{Ern}.

\begin{lemma} \label{lem:adjointness-continuity}
Let $f \colon X \to Y$ be an adjointable map between orthosets. Then we have:
\begin{itemize}

\item[\rm (i)] $f(0) = 0$.

\item[\rm (ii)] For any $A \subseteq X$, we have $f(A)\cc = f(A\cc)\cc$. Consequently,
\[ f(A\cc) \;\subseteq\; f(A)\cc \]
and in particular, $f(\{x_1,x_2\}\cc) \subseteq \{f(x_1),f(x_2)\}\cc$ for any $x_1, x_2 \in X$.

\item[\rm (iii)] If $A \subseteq Y$ is orthoclosed, so is $f^{-1}(A)$.
\end{itemize}
\end{lemma}

\begin{proof}
Let $g$ be an adjoint of $f$.

Ad (i): $0 \perp g(f(0))$ implies $f(0) \perp f(0)$.

Ad (ii): For any $y \in Y$, we have $y \perp f(A)$ iff $g(y) \perp A$ iff $g(y) \perp A\cc$ iff $y \perp f(A\cc)$. Hence $f(A)\c = f(A\cc)\c$ and the assertions follow.

Ad (iii): If $A \in {\mathsf C}(Y)$, then $f(f^{-1}(A)\cc) \subseteq f(f^{-1}(A))\cc \subseteq A\cc = A$ by part (ii). Hence $f^{-1}(A)\cc \subseteq f^{-1}(A)$, that is, $f^{-1}(A) \in {\mathsf C}(X)$.
\end{proof}

We wish to relate maps between orthosets to maps between the associated ortholattices. We start with a lemma on adjointable maps between ortholattices.

\begin{lemma} \label{lem:adjointable-maps-between-ortholattices}
Let $h \colon L \to M$ be a map between complete ortholattices. Assume that $h$, viewed as a map between orthosets, is adjointable. Then $h$ preserves the order. In fact, $h$ is sup-preserving.
\end{lemma}

\begin{proof}
Let $k$ be an adjoint of $h$ and let $a_\iota \in L$, $\iota \in I$. Then for any $b \in M$, we have $h(\bigvee_\iota a_\iota) \perp b$ iff $\bigvee_\iota a_\iota \perp k(b)$ iff $a_\iota \perp k(b)$ for all $\iota \in I$ iff $h(a_\iota) \perp b$ for all $\iota \in I$ iff $\bigvee_\iota h(a_\iota) \perp b$. This shows that $h(\bigvee_\iota a_\iota) = \bigvee_\iota h(a_\iota)$.
\end{proof}

Given a map $f \colon X \to Y$ between orthosets, we define as follows the induced map between the associated ortholattices:
\begin{equation} \label{fml:Cf}
{\mathsf C}(f) \colon {\mathsf C}(X) \to {\mathsf C}(Y) \komma A \mapsto f(A)\cc.
\end{equation}
The following lemma shows that if $f$ is adjointable, so is ${\mathsf C}(f)$. Moreover, ${\mathsf C}(f)$ preserves arbitrary joins and its lattice adjoint is expressible by means of the orthoset adjoint.

\begin{lemma} \label{lem:lattice-adjoint}
Let $f \colon X \to Y$ and $g \colon Y \to X$ be an adjoint pair of maps between orthosets. Then the following holds:
\begin{itemize}
\item[\rm (i)] Seen as maps between orthosets, ${\mathsf C}(f)$ and ${\mathsf C}(g)$ are an adjoint pair. That is, for any $A \in {\mathsf C}(X)$ and $B \in {\mathsf C}(Y)$,
\begin{equation} \label{fml:OL-adjoint}
{\mathsf C}(f)(A) \perp B \quad\text{if and only if}\quad A \perp {\mathsf C}(g)(B).
\end{equation}
\item[\rm (ii)] ${\mathsf C}(f)$ is sup-preserving. That is, for any $A_\iota \in {\mathsf C}(X)$, $\iota \in I$,
\[ \textstyle f(\bigvee_\iota A_\iota)\cc \;=\; \bigvee_\iota f(A_\iota)\cc. \]
\item[\rm (iii)] For any $A \in {\mathsf C}(X)$ and $B \in {\mathsf C}(Y)$,
\begin{equation} \label{fml:lattice-adjoint}
f(A)\cc \subseteq B \quad\text{if and only if}\quad A \subseteq g(B\c)\c.
\end{equation}
\end{itemize}
\end{lemma}

\begin{proof}
Ad (i): We have $f(A)\cc \perp B$ iff $f(A) \perp B$ iff $A \perp g(B)$ iff $A \perp g(B)\cc$.

Ad (ii): This is clear from part (i) and Lemma~\ref{lem:adjointable-maps-between-ortholattices}.

Ad (iii): This is a reformulation of part (i).
\end{proof}

We shall now discuss the injectivity and surjectivity of adjointable maps.

We define the {\it kernel} and the {\it range} of a map $f \colon X \to Y$, respectively, by
\begin{align*}
\kernel f \;=\; & \{ x \in X \colon f(x) = 0 \}, \\
\image f \;=\; & \{ f(x) \colon x \in X \}.
\end{align*}
We say that $f$ has a {\it zero kernel} if $\kernel f = \{0\}$.

\begin{lemma} \label{lem:kernels-images}
Let $f \colon X \to Y$ and $g \colon Y \to X$ be an adjoint pair of maps between orthosets. Then the following holds:
\begin{itemize}

\item[\rm (i)] $\kernel f = (\image g)\c = g(\image f)\c$ and $\kernel g = (\image f)\c = f(\image g)\c$.

\item[\rm (ii)] $(\image f)\cc = f(\image g)\cc = f((\kernel f)\c)\cc$ and $(\image g)\cc = g(\image f)\cc = g((\kernel g)\c)\cc$.

\item[\rm (iii)] Assume that $\image f$ is orthoclosed. Assume moreover that $X$ is an atomistic Dacey space, ${\mathsf C}(X)$ has the covering property, and $Y$ is Fr\' echet. Then $\image f = f((\kernel f)\c)$.

\end{itemize} 
\end{lemma}

\begin{proof}
Ad (i): For any $x \in X$, we have $f(x) = 0$ iff $f(x) \perp Y$ iff $x \perp g(Y)$. Similarly, for any $x \in X$ we have $f(x) = 0$ iff $f(x) \perp f(X)$ iff $x \perp g(f(X))$. This shows the first two equalities and the remaining ones hold by symmetry.

Ad (ii): This follows from part (i).

Ad (iii): Let $x \in X$. We have to show that there is a $y \perp \kernel f$ such that $f(x) = f(y)$. This is clear if $x \in \kernel f$ or $x \perp \kernel f$. Assume that $x \notin \kernel f$ and $x \notperp \kernel f$. Note that then $\kernel f$ is neither $\{0\}$ nor $X$. By Lemma~\ref{lem:atomistic}, $\{x\}\cc = \class x \cup \{0\}$ is an atom of ${\mathsf C}(X)$. As ${\mathsf C}(X)$ is an AC orthomodular lattice, there are by \cite[Lemma~(30.7)]{MaMa} proper elements $y \in (\kernel f)\c$ and $z \in \kernel f$ such that $\{x\}\cc \subseteq \{y,z\}\cc$. By Lemmas~\ref{lem:lattice-adjoint}(ii) and~\ref{lem:Frechet}, we have $\{f(x),0\} = f(\class x \cup \{0\}) = f(\{x\}\cc) \subseteq f(\{y,z\}\cc) \subseteq f(\{y\}\cc)\cc \vee f(\{z\}\cc)\cc = f(\class y \cup \{0\})\cc \vee f(\class z \cup \{0\})\cc = \{f(y),0\}$ and hence $f(x) = f(y)$ as desired.
\end{proof}

\begin{lemma} \label{lem:injective-surjective}
Let $f \colon X \to Y$ and $g \colon Y \to X$ be an adjoint pair of maps between orthosets. The following statement {\rm (a)} implies {\rm (b)}, and {\rm (b)} implies {\rm (c)}:
\begin{itemize}

\item[\rm (a)] $f$ is injective and $\image g$ is orthoclosed.

\item[\rm (b)] $f$ has a zero kernel and $\image g$ is orthoclosed.

\item[\rm (c)] $g$ is surjective.

\end{itemize}
If $X$ is irredundant, {\rm (a)}, {\rm (b)}, and {\rm (c)} are pairwise equivalent.
\end{lemma}

\begin{proof}
Clearly, (a) implies (b). Moreover, if $\kernel f = \{0\}$ and $\image g \in {\mathsf C}(Y)$, then $\image g = (\kernel f)\c = X$ by Lemma~\ref{lem:kernels-images}(i), that is, $g$ is surjective. Hence (b) implies (c).

Assume that $X$ is irredundant and $g$ is surjective. Let $x_1, x_2 \in X$ be such that $f(x_1) = f(x_2)$. For any $x \in X$, $x_1 \perp x$ implies that $x_1 \perp g(y)$ for some $y \in Y$ such that $g(y) = x$, hence $f(x_2) = f(x_1) \perp y$, and $x_2 \perp g(y) = x$. Similarly, $x_2 \perp x$ implies $x_1 \perp x$, hence we have $x_1 \prl x_2$. By irredundancy, we conclude $x_1 = x_2$, and it follows that $f$ is injective.
\end{proof}

Let $f \colon X \to Y$ be adjointable. Restricting the domain of $f \colon X \to Y$ to the subspace $(\kernel f)\c$ of $X$ and the codomain to the subspace $(\image f)\cc$ of $Y$, we get the map
\[ \zerokernel{f} \colon (\kernel f)\c \to (\image f)\cc \komma x \mapsto f(x), \]
which has a zero kernel. We call $\zerokernel{f}$ the {\it zero-kernel restriction} of $f$. Let $g \colon Y \to X$ be an adjoint of $f$. By Lemma~\ref{lem:kernels-images}(i), $((\image g)\cc, \kernel f)$ is a decomposition of $X$ and $((\image f)\cc, \kernel g)$ is a decomposition of $Y$. Moreover, $\zerokernel f$ and $\zerokernel g$ form an adjoint pair of maps between the subspace $(\image g)\cc$ of $X$ and the subspace $(\image f)\cc$ of $Y$.

\begin{lemma} \label{lem:kernel-and-image}
Let $f \colon X \to Y$ and  $g \colon Y \to X$ be an adjoint pair of maps between orthosets. Assume that $\image f$ and $\image g$ are orthoclosed. Then $(\image g, \kernel f)$ is a decomposition of $X$, $(\image f, \kernel g)$ is a decomposition of $Y$, and $\zerokernel f$ and $\zerokernel g$ form an adjoint pair of maps between the subspaces $\image g$ and $\image f$.

Assume, in addition, that $X$ and $Y$ are Fr\' echet Dacey spaces and that ${\mathsf C}(X)$ and ${\mathsf C}(Y)$ have the covering property. Then $\zerokernel f$ and $\zerokernel g$ are bijections.
\end{lemma}

\begin{proof}
The first part clear from the preceding remarks.

Under the additional assumptions, $\zerokernel f$ and $\zerokernel g$ are surjective by Lemma~\ref{lem:kernels-images}(iii). Moreover, $\image f$ and $\image g$ are irredundant by Lemma~\ref{lem:subspaces-of-Dacey-spaces}(iv), hence $\zerokernel f$ and $\zerokernel g$ are injective by Lemma~\ref{lem:injective-surjective}.
\end{proof}

Let us finally discuss subspaces that are, together with their orthocomplement, invariant for a given map.

Let $f \colon X \to X$ be a map of an orthoset to itself. We call a subspace $A$ of $X$ {\it reducing} for $f$ if $f(A) \subseteq A$ and $f(A\c) \subseteq A\c$. In case when every subspace $A$ of $X$ is reducing for $f$, we call $f$ {\it scalar}.

\begin{lemma} \label{lem:reducing-subspaces}
Let $X$ be an orthoset and let $f \colon X \to X$ be adjointable.
\begin{itemize}

\item[\rm (i)] Let $A \in {\mathsf C}(X)$ and let $g \colon X \to X$ be an adjoint of $f$. The following are equivalent:
\begin{itemize}

\item[\rm (a)] $A$ is reducing for $f$;

\item[\rm (b)] $A$ is reducing for $g$;

\item[\rm (c)] $f(A) \subseteq A$ and $g(A) \subseteq A$.
\end{itemize}

\item[\rm (ii)] The set $\mathsf R$ of all subspaces of $X$ that are reducing for $f$ is closed under arbitrary meets and joins as well as the orthocomplementation. In particular, $\mathsf R$ is a subortholattice of ${\mathsf C}(X)$.

\end{itemize}
\end{lemma}

\begin{proof}
Ad (i): We have $f(A\c) \subseteq A\c$ iff $f(A\c) \perp A$ iff $A\c \perp g(A)$ iff $g(A) \subseteq A$. Similarly, we see that $f(A) \subseteq A$ iff $g(A\c) \subseteq A\c$. The asserted equivalences follow.

Ad (ii): Let $A_\iota \in \mathsf R$, $\iota \in I$. Then obviously $f(\bigcap_\iota A_\iota) \subseteq \bigcap_\iota A_\iota$, and we have $f(\bigvee_\iota A_\iota) \subseteq \bigvee_\iota A_\iota$ by Lemma~\ref{lem:lattice-adjoint}(ii). In view of the De Morgan laws, we conclude that $\bigcap_\iota A_\iota, \bigvee_\iota A_\iota \in \mathsf R$. Moreover, $A \in \mathsf R$ clearly implies $A\c \in \mathsf R$.
\end{proof}

\begin{lemma} \label{lem:scalar-map}
Let $X$ be an orthoset and let $f \colon X \to X$ be adjointable. If $f \prl \id_X$, then $f$ is scalar. In case when $X$ is atomistic and irreducible and $f$ is not constant $0$, then also the converse holds.

\end{lemma}

\begin{proof}
Let $f \prl \id_X$. For any $A \in {\mathsf C}(X)$, we have $f(A) \subseteq \bigcup \{ \class{f(x)} \colon x \in A \} = \bigcup \{ \class x \colon x \in A \} = A$. Hence $f$ is scalar.

Assume now that $X$ is atomistic and irreducible, and let $f \colon X \to X$ be scalar and not constant $0$. For $x \in X$, we then have $f(x) \in f(\{x\}\cc) \subseteq \{x\}\cc = \class x \cup \{0\}$. Hence either $f(x) = 0$ or $f(x) \prl x$. Consequently, $x \notin \kernel f$ implies $x \in (\image f)\cc = f((\kernel f)\c)\cc \subseteq (\kernel f)\c$ by Lemma~\ref{lem:kernels-images}(ii), and it follows $X = \kernel f \cup (\kernel f)\c$. As $X$ is irreducible and $\kernel f \neq X$, we conclude $\kernel f = \{0\}$. This shows that $f \prl \id_X$.
\end{proof}

\section{Orthometric correspondences}
\label{sec:Orthometric-correspondences}

An obvious condition to consider when discussing maps between orthosets is the preservation of the orthogonality relation. The concept of adjointability of maps, which we introduced in the previous section, might seem at first sight to be unrelated to this issue. This is, however, a mistaken view: a bijective map preserves the orthogonality relation in both directions exactly if its inverse is an adjoint.

Apart from bijective correspondences, we study in this section partial orthometries. Our definition is chosen in analogy to partial isometries between Hilbert spaces. Basic examples include the inclusion maps $\iota \colon A \to X$, where $A$ is a subspace of an orthoset $X$, as well as the Sasaki maps, which are adjoints of inclusion maps. This issue will be of particular importance in the sequel, especially for the description of Dacey spaces.

We say that a map $f \colon X \to Y$ between orthosets {\it preserves} the orthogonality relation if, for any $x_1, x_2 \in X$, $x_1 \perp x_2$ implies $f(x_1) \perp f(x_2)$. We say that $f$ {\it reflects} $\perp$ if, for any $x_1, x_2 \in X$, $f(x_1) \perp f(x_2)$ implies $x_1 \perp x_2$.

\begin{lemma} \label{lem:perp-preserving-reflecting}
Let $f \colon X \to Y$ and $g \colon Y \to X$ be an adjoint pair of maps between orthosets. Then $f$ preserves and reflects $\perp$ if and only if $g \circ f \prl \id$.
\end{lemma}

\begin{proof}
For $f$ to preserve and reflect $\perp$ means that, for any $x, x' \in X$, $x \perp x'$ is equivalent to $f(x) \perp f(x')$. Furthermore, $g \circ f \prl \id$ means that, for any $x, x' \in X$, $x \perp x'$ is equivalent to $g(f(x)) \perp x'$. The assertion follows.
\end{proof}

Note that an orthoisomorphism between orthosets is the same as a bijection preserving and reflecting $\perp$.

\begin{proposition} \label{prop:unitary-maps}
Given a map $f \colon X \to Y$ between orthosets, the following are equivalent:
\begin{itemize}

\item[\rm (a)] $f$ is an orthoisomorphism;

\item[\rm (b)] $f$ is bijective and $f^{-1}$ is an adjoint of $f$;

\item[\rm (c)] $f$ is bijective and possesses an adjoint $g \colon Y \to X$ such that $g \circ f \prl \id_X$.

\item[\rm (d)] $f$ is bijective and adjointable, and $g \circ f \prl \id_X$ for any adjoint $g$ of $f$.

\end{itemize}
\end{proposition}

\begin{proof}
(a) $\Rightarrow$ (b): Let $f$ be an orthoisomorphism. Then, for any $x \in X$ and $y \in Y$, we have $f(x) \perp y$ iff $x \perp f^{-1}(y)$. Hence $f^{-1}$ is an adjoint of $f$.

(b) $\Rightarrow$ (c) and (d) $\Rightarrow$ (c): These implications hold trivially.

(c) $\Rightarrow$ (d): This is clear from Lemma~\ref{lem:irredundancy-adjoints-1}(i).

(c) $\Rightarrow$ (a): This is clear from Lemma~\ref{lem:perp-preserving-reflecting}.
\end{proof}

An orthoisomorphism from an orthoset to itself is called an {\it orthoautomorphism}.

\begin{lemma} \label{lem:C-of-automorphism}
Let $X$ be an orthoset.
\begin{itemize}

\item[\rm (i)] For any orthoautomorphism $f$ of an orthoset $X$, $\,{\mathsf C}(f) \colon {\mathsf C}(X) \to {\mathsf C}(X) \komma A \mapsto f(A)$ is an automorphism of ${\mathsf C}(X)$. The assignment $f \mapsto {\mathsf C}(f)$ defines a homomorphism from the group of orthoautomorphisms of $X$ to the group of automorphisms of ${\mathsf C}(X)$. Its kernel consists exactly of the scalar orthoautomorphisms.

\item[\rm (ii)] Let $f \colon X \to X$ be adjointable. Then $f$ is a scalar orthoautomorphism if and only if $f$ is bijective and $f \prl \id_X$.

\end{itemize}
\end{lemma}

\begin{proof}
Ad (i): Only the last assertion might need a comment. Let $f$ be an orthoautomorphism of $X$. Then ${\mathsf C}(f) = \id$ if and only if $f(A) = A$ for any $A \in {\mathsf C}(X)$. Thus any subspace is in this case reducing and $f$ is scalar. Conversely, if $f$ is scalar, we have $f(A) \subseteq A$ for any $A \in {\mathsf C}(X)$. By Proposition~\ref{prop:unitary-maps}, $f^{-1}$ is an adjoint of $f$, hence by Lemma~\ref{lem:reducing-subspaces} we also have $f^{-1}(A) \subseteq A$ and consequently $A \subseteq f(A)$ for any $A \in {\mathsf C}(X)$. That is, ${\mathsf C}(f) = \id$.

Ad (ii): Assume that $f$ is a scalar orthoautomorphism. By part (i), $f(\{y\}\c) = \{y\}\c$ and $f^{-1}(\{y\}\c) = \{y\}\c$ for any $y \in X$. For $x \in X$, we hence have that $x \perp y$ implies $f(x) \perp y$, which in turn implies $x \perp y$. That is, $f(x) \prl x$, and we conclude $f \prl \id_X$.

Conversely, assume that $f$ is bijective and $f \prl \id_X$. We then have for any $x, y \in X$ that $x \perp y$ iff $f(x) \perp y$ iff $f(x) \perp f(y)$, that is, $f$ is an orthoautomorphism. Moreover, $f(\class x) = \class x$ for any $x \in X$ and hence $f(A) = A$ for any $A \in {\mathsf C}(X)$, that is, $f$ is scalar.
\end{proof}

Let us now consider a more general type of orthometric correspondence. Let $f \colon X \to Y$ and $g \colon Y \to X$ be an adjoint pair such that $\image f$ and $\image g$ are orthoclosed. By Lemma~\ref{lem:kernel-and-image}, $X$ decomposes into $(\image g,\kernel f)$, $\,Y$ decomposes into $(\image f,\kernel g)$, and $f$ and $g$ restrict to the adjoint pair of map $\zerokernel f$ and $\zerokernel g$ between $\image g$ and $\image f$. We consider the case that these maps are orthoisomorphisms.

We call a map $f \colon X \to Y$ a {\it partial orthometry} if $f$ possesses an adjoint $g \colon Y \to X$ such that the following holds: there are subspaces $A$ of $X$ and $B$ of $Y$ such that $A\c = \ker f$, $B\c = \ker g$, and $f$ and $g$ establish mutually inverse orthoisomorphisms between $A$ and $B$. In this case, we call $g$ a {\it generalised inverse} of $f$. Clearly, in this case also $g$ is a partial orthometry, and $f$ is a generalised inverse of $g$.

For a map $f \colon X \to Y$ and sets $A \subseteq X$ and $\image f \subseteq B \subseteq Y$, we will denote by $f|_A^B$ the map $f$ restricted to $A$ and corestricted to $B$.

\begin{proposition} \label{prop:partial-isometry}
Let $f \colon X \to Y$ and $g \colon Y \to X$ be an adjoint pair of maps between orthosets. Then the following are equivalent:
\begin{itemize}

\item[\rm (a)] $f$ is a partial orthometry and $g$ is a generalised inverse of $f$.

\item[\rm (b)] $\image g = (\ker f)\c$, $\image f = (\ker g)\c$, and $f|_{\image g}^{\image f}$ is an orthoisomorphism between the subspaces $\image g$ and $\image f$, whose inverse is $g|_{\image f}^{\image g}$.

\item[\rm (c)] $\zerokernel f$ and $\zerokernel g$ are mutually inverse bijections.

\item[\rm (d)] $\image f$ and $\image g$ are orthoclosed, and $f \circ g \circ f = f$ as well as $g \circ f \circ g = g$.

\end{itemize}
In this case, $\zerokernel f = f|_{\image g}^{\image f}$ and $\zerokernel g = g|_{\image f}^{\image g}$.
\end{proposition}

\begin{proof}
(a) $\Rightarrow$ (d): Let $f$, $g$, $A$, and $B$ as indicated in the definition of a partial orthometry. Then $B \subseteq \image f \subseteq (\ker g)\c = B$ by Lemma~\ref{lem:kernels-images}(i), that is, $B = \image f$. Similarly, we see that $A = \image g$. In particular, $\image f$ and $\image g$ are orthoclosed and we have $f \circ g \circ f = f$ and $g \circ f \circ g = g$.

(d) $\Rightarrow$ (c): Let (d) hold. By Lemma~\ref{lem:kernel-and-image}, $\zerokernel f$ and $\zerokernel g$ are maps between $\image g$ and $\image f$. Moreover, $g \circ f \circ g = g$ and $f \circ g \circ f = f$ imply that $\zerokernel f$ and $\zerokernel g$ are mutually inverse bijections.

(c) $\Rightarrow$ (b): Assuming (c), we have that $\zerokernel f$ and $\zerokernel g$ are mutually inverse bijections between $(\image g)\cc$ and $(\image f)\cc$. It follows that $\image f$ and $\image g$ are orthoclosed and hence $\image g = (\ker f)\c$ and $\image f = (\ker g)\c$. By Lemma~\ref{lem:kernel-and-image}, $\zerokernel g = g|_{\image f}^{\image g}$ is an adjoint of $\zerokernel f = f|_{\image g}^{\image f}$. By Proposition~\ref{prop:unitary-maps}, $\zerokernel f$ is an orthoisomorphism. This shows (b) as well as the last assertion.

(b) $\Rightarrow$ (a): This is obvious.
\end{proof}

We call an injective partial orthometry an {\it orthometry}. That is, $f \colon X \to Y$ is an orthometry if there is a subspace $B$ of $Y$ and $f$ possesses an adjoint $g \colon Y \to X$ such that $\kernel g = B\c$, and $f|^B$ and $g|_B$ are mutually inverse orthoisomorphisms between $X$ and $B$. Similarly, we call a surjective partial orthometry a {\it coorthometry}. That is, $f \colon X \to Y$ is a coorthometry if there is a subspace $A$ of $X$ and $f$ possesses an adjoint $g \colon Y \to X$ such that $\kernel f = A\c$, and $f|_A$ and $g|^A$ are mutually inverse orthoisomorphisms between $A$ and $Y$. Clearly, a generalised inverse of an orthometry is a coorthometry and vice versa. Note also that the generalised inverse of a coorthometry is uniquely determined. Finally, we may mention that a bijective partial isometry is the same as an orthoisomorphism.

\begin{proposition} \label{prop:isometry-coisometry}
Let $f \colon X \to Y$ and $g \colon Y \to X$ be an adjoint pair of maps between orthosets. Then the following are equivalent:
\begin{itemize}

\item[\rm (a)] $f$ is an orthometry and $g$ is a generalised inverse of $f$.

\item[\rm (b)] $g$ is coorthometry and $f$ is a generalised inverse of $g$.

\item[\rm (c)] $\image f = (\kernel g)\c$, and $f|^{\image f}$ is an orthoisomorphism between $X$ and the subspace $\image f$ of $Y$, whose inverse is $g|_{\image f}$.

\item[\rm (d)] $\image f$ is orthoclosed and $g \circ f = \id_X$.

\end{itemize}
In this case, $\zerokernel f = f|^{\image f}$ and $\zerokernel g = g|_{\image f}$.
\end{proposition}

\begin{proof}
Assume that $f$ is a partial orthometry and $g$ is a generalised inverse of $f$. Then $f$ is an orthometry if and only if $g$ is a coorthometry if and only if $\kernel f = \{0\}$. Hence the assertions follow from Proposition~\ref{prop:partial-isometry}.
\end{proof}

For Dacey spaces, we may characterise partial orthometries without reference to a specific adjoint.

\begin{proposition} \label{prop:partial-isometry-between-Dacey-spaces}
Let $f \colon X \to Y$ be an adjointable map between Dacey spaces. Then the following holds:
\begin{itemize}

\item[\rm (i)] $f$ is a partial orthometry if and only if there are subspaces $A$ of $X$ and $B$ of $Y$ such that $f(x) = 0$ if $x \perp A$, and $f$ establishes an orthoisomorphism between $A$ and $B$. In this case, $A = (\kernel f)\c$, $\,B = \image f$, and $\zerokernel f = f|_A^B$.

\item[\rm (ii)] $f$ is an orthometry if and only if $B = \image f$ is orthoclosed and $f|^B$ is an orthoisomorphism between $X$ and $B$. In this case, $B = \image f$ and $\zerokernel f = f|^B$.

\item[\rm (iii)] $f$ is a coorthometry if and only if there is a subspace $A$ of $X$ such that $f(x) = 0$ if $x \perp A$, and $f|_A$ is an orthoisomorphism between $A$ and $Y$. In this case, $A = (\kernel f)\c$ and $\zerokernel f = f|_A$.

\end{itemize}
\end{proposition}

\begin{proof}
Ad (i): The ``only if'' part holds by definition, hence we only have to show the ``if'' part.

Let $f$, $A$, and $B$ as indicated. Let $g$ be an adjoint of $f$. We have $A\c \subseteq \kernel f$ and $A \cap \kernel f = \{0\}$, hence $A\c = \kernel f$ by orthomodularity and $\image g \subseteq (\kernel f)\c = A$. Furthermore, $B \subseteq \image f$ implies $\kernel g = (\image f)\c \subseteq B\c$, and $f(A) = B$ means $f(A) \perp B\c$, hence $A \perp g(B\c) \subseteq A$, that is, $B\c \subseteq \kernel g$. Hence $B\c = \kernel g$. In particular, $\zerokernel f = f|_A^B$ and $\zerokernel g = g|_B^A$. We also have that $\image f \subseteq (\kernel g)\c = B \subseteq \image f$, that is, $B = \image f$.

By assumption, $\zerokernel f$ is an orthoisomorphism. As both ${\zerokernel f}^{-1}$ and $\zerokernel g$ are adjoints of $\zerokernel f$, we have ${\zerokernel f}^{-1} \prl \zerokernel g$ by Lemma~\ref{lem:irredundancy-adjoints-1}(i). By Lemma~\ref{lem:subspaces-of-Dacey-spaces}(iii), ${\zerokernel f}^{-1}(y) \prl g(y)$, where $y \in B$, also holds in $X$. We put $\tilde g \colon Y \to X \komma \footnotesize y \mapsto \begin{cases} {\zerokernel f}^{-1}(y) & \text{if $y \in B$}, \\ g(y) & \text{otherwise.} \end{cases}$ Then $\tilde g \prl g$, hence $\tilde g$ is an adjoint of $f$ as well, $\kernel \tilde g = \kernel g = B\c$, and $\tilde g|_B^A = (f|_A^B)^{-1}$. Now it is clear that $f$ is a partial orthometry.

Parts (ii) and (iii) follow as special cases from part (i).
\end{proof}

We note that the discussed properties of maps between orthosets are preserved by the transition to irredundant quotients.

\begin{lemma}
Let $f \colon X \to Y$ be a partial orthometry (orthometry, coorthometry, orthoisomorphism) between orthosets. Then so is $P(f) \colon P(X) \to P(Y)$.
\end{lemma}

\begin{proof}
Let $f$ be a partial orthometry. Then so is $P(f)$ by criterion (d) of Proposition~\ref{prop:partial-isometry}. Moreover, if $f$ is injective, then $f$ has a zero kernel. Hence the partial orthometry $P(f)$ has likewise a zero kernel, which means that $P(f)$ is injective. Finally, if $f$ is surjective, so is $P(f)$.
\end{proof}

If we deal with an orthoset $X$ and a subspace $A$ of $X$, a more particular terminology seems to be in order. We refer to $\iota \colon A \to X \komma a \mapsto a$ as the {\it inclusion map} of $A$ (into $X$). If $\iota$ is an orthometry, then we call a generalised inverse of $\iota$ a {\it Sasaki map} (onto $A$).

\begin{lemma} \label{lem:isometric-inclusion}
Let $A$ be a subspace of the orthoset $X$ and let $\iota \colon A \to X$ be the inclusion map.
\begin{itemize}

\item[\rm (i)] For a map $\sigma \colon X \to A$, the following are equivalent:

\begin{itemize}

\item[\rm (a)] $\sigma$ is an adjoint of $\iota$.

\item[\rm (b)] For any $x \in X$,
\begin{equation} \label{fml:Sasaki-map}
\{ a \in A \colon a \perp x \} \;=\; \{ a \in A \colon a \perp \sigma(x) \}.
\end{equation}

\item[\rm (c)] For any $x \in X$,
\[ A\c \vee \{x\}\cc \;=\; A\c \vee \{\sigma(x)\}\cc. \]
\end{itemize}

\item[\rm (ii)] $\iota$ is an orthometry if and only if $\iota$ is adjointable. In this case, any adjoint of $\iota$ is equivalent to a Sasaki map onto $A$.

\item[\rm (iii)] A map $\sigma \colon X \to A$ is a Sasaki map if and only if $\sigma$ is an adjoint of $\iota$ such that $\sigma|_A = \id_A$.

\end{itemize}
\end{lemma}

\begin{proof}
Ad (i): The following statements are equivalent: (a) holds; for any $a \in A$ and $x \in X$, $a \perp x$ is equivalent to $a \perp \sigma(x)$; (b) holds; for any $x \in X$, $\,\{x\}\c \cap A = \{\sigma(x)\}\c \cap A$; (c) holds.

Ad (ii): An orthometry is by definition adjointable. Conversely, assume that $\sigma$ is an adjoint of $\iota$. For any $a \in A$, we have $\{\sigma(a)\}\ce{A} = \{a\}\ce{A}$ by (\ref{fml:Sasaki-map}), that is, $\sigma(a) \prl a$ in $A$. We put $\tilde \sigma \colon X \to A \komma \footnotesize x \mapsto \begin{cases} x & \text{if $x \in A$}, \\ \sigma(x) & \text{otherwise.} \end{cases}$ Then $\tilde\sigma \prl \sigma$ and hence also $\tilde\sigma$ is an adjoint of $\iota$. Moreover, $\kernel \tilde\sigma = \kernel \sigma = (\image \iota)\c = A\c$, and $\tilde\sigma|_A = \id_A$. We conclude that $\iota$ is an orthometry and $\tilde\sigma$ is a generalised inverse of $\iota$.

Ad (iii): The ``only if'' part holds by definition and the ``if'' part follows from Proposition~\ref{prop:isometry-coisometry}, criterion (d).
\end{proof}

Let $A$ still be a subspace of an orthoset $X$. A partial orthometry $p \colon X \to X$ such that $\zerokernel p = \id_A$ is called a {\it projection} (onto $A$).

\begin{lemma} \label{lem:projections}
Let $X$ be an orthoset and let $p \colon X \to X$ be adjointable. The following are equivalent:
\begin{itemize}

\item[\rm (a)] $p$ is a projection.

\item[\rm (b)] $p = \iota \circ \sigma$, where $\iota \colon A \to X$ is the inclusion map of a subspace $A$ and $\sigma$ is a Sasaki map onto $A$.

\item[\rm (c)] $p$ is idempotent and self-adjoint, and $\image p$ is orthoclosed.

\end{itemize}
\end{lemma}

\begin{proof}
(a) $\Rightarrow$ (b): Assume that $A \in {\mathsf C}(X)$ and $p$ is a partial isometry such that $\zerokernel p = p|_{(\kernel p)\c}^{\image p} = \id_A$. Let $\iota \colon A \to X$ be the inclusion map. As $\image p = A$, we may define $\sigma = p|^A$. Then $p = \iota \circ \sigma$. We claim that $\sigma$ is a Sasaki map onto $A$. Indeed, we have $\sigma|_A = p|_A^A = \id_A$. Furthermore, let $q$ be a generalised inverse of $p$. Then $\zerokernel q = (\zerokernel p)^{-1} = \id_A$. Consequently, for any $a \in A$ and $x \in X$, we have $a \perp x$ iff $q(a) \perp x$ iff $a \perp p(x)$ iff $a \perp \sigma(x)$. The claim now follows by Lemma~\ref{lem:isometric-inclusion}(iii).

(b) $\Rightarrow$ (c): Let $A \in {\mathsf C}(X)$ and $p = \iota \circ \sigma$ as indicated. As $\iota$ and $\sigma$ are an adjoint pair, it follows that $p$ is self-adjoint. Moreover, $\image p = A$ is orthoclosed. Finally, as $p$ is on $A$ the identity, $p$ is idempotent.

(c) $\Rightarrow$ (a): Assume (c) and let $A = \image p$. Then $p(p(x)) = p(x)$ for any $x \in X$ and hence $p|_A = \id_A$. Moreover, by Lemma~\ref{lem:kernels-images}(i), $(\kernel p)\c = (\image p)\cc = A$ and hence $\zerokernel p = p|_A^A = \id_A$. According to criterion (c) of Proposition~\ref{prop:partial-isometry}, $p$ is a partial isometry.
\end{proof}

If an orthoset $X$ is such that all inclusion maps are adjointable, then, as we see next, $X$ is a Dacey space. This fact was exploited in \cite{LiVe} in order to characterise orthosets associated with orthomodular spaces.

\begin{theorem} \label{thm:adjoints-and-Dacey}
Let $X$ be an orthoset such that, for any subspace $A$ of $X$, the inclusion map $\iota \colon A \to X$ is adjointable. Then the following holds.
\begin{itemize}

\item[\rm (i)] $X$ is Dacey.

\item[\rm (ii)] If $X$ is atomistic, ${\mathsf C}(X)$ has the covering property.

\end{itemize}
\end{theorem}

\begin{proof}
Ad (i): Let $A \in {\mathsf C}(X)$ and let $D$ be a maximal $\perp$-set contained in $A$. Let $\sigma$ be an adjoint of the inclusion map of $D\c$ into $X$. Assume that there is some $e \in A \setminus D\cc$. By Lemma~\ref{lem:isometric-inclusion}(i), $D\cc \subsetneq D\cc \vee \{e\}\cc = D\cc \vee \{\sigma(e)\}\cc$. But then $D \cup \{\sigma(e)\} \subseteq A$ is a $\perp$-set, a contradiction. Hence $D\cc = A$ and we conclude from Lemma~\ref{lem:Dacey-space} that ${\mathsf C}(X)$ is orthomodular.

Ad (ii): Let $X$ be atomistic. Then, by Lemma~\ref{lem:atomistic}, ${\mathsf C}(X)$ is atomistic, the atoms being $\{x\}\cc$, $x \in X$. Let $A \in {\mathsf C}(X)$ and $x \notin A$. By Lemma~\ref{lem:isometric-inclusion}(i), there is a $y \perp A$ such that $A \vee \{x\}\cc = A \vee \{y\}\cc$. As $\{y\}\cc$ is an atom of the orthomodular lattice ${\mathsf C}(X)$, it follows that $A$ is covered by $A \vee \{x\}\cc$.
\end{proof}

We note that the converse of Theorem~\ref{thm:adjoints-and-Dacey}(i) does not hold: a Dacey space does not in general have the property that inclusion maps of subspaces are adjointable. Indeed, consider any complete atomistic orthomodular lattice $L$ that does not have the covering property. For instance, let $L$ be the horizontal sum of the $4$-element Boolean algebra and the $8$-element Boolean algebra (see Figure~\ref{fig:OML}). By Proposition~\ref{prop:Frechet-orthosets}, $L$ is isomorphic to ${\mathsf C}({\mathsf B}(L))$, hence the atomistic Dacey space ${\mathsf B}(L)$ provides by Theorem~\ref{thm:adjoints-and-Dacey}(ii) a counterexample. In contrast, we will see below (Lemma~\ref{lem:inclusion-maps-in-OML-as-OS}) that, for any complete orthomodular lattice~$L$, $L\supOS$ does have the property that inclusion maps of subspaces are adjointable.
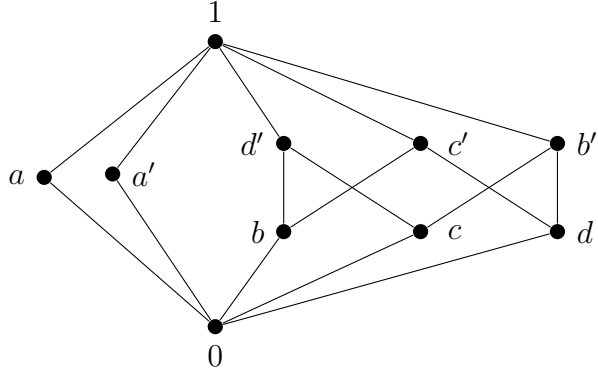
\begin{figure}[h]
\begin{center}
\begin{tikzpicture}[scale=0.9]
\node[circle,fill,inner sep=0,minimum size=2mm,label={left:$a$}] (a1) at (-2.5,0.5) {};

\node[circle,fill,inner sep=0,minimum size=2mm,label={right:$a'$}] (b1) at (-1.5,0.55) {};

\node[circle,fill,inner sep=0,minimum size=2mm,label={left:$d'$}] (ab2) at (1,1) {};

\node[circle,fill,inner sep=0,minimum size=2mm,label={right:$\ c'$}] (ac2) at (3,1) {};
\node[circle,fill,inner sep=0,minimum size=2mm,label={right:$b'$}] (bc2) at (5,1) {};
\node[circle,fill,inner sep=0,minimum size=2mm,label={left:$b$}] (a2) at (1,-0.3) {};
\node[circle,fill,inner sep=0,minimum size=2mm,label={right:$\ c$}] (b2) at (3,-0.3) {};
\node[circle,fill,inner sep=0,minimum size=2mm,label={right:$d$}] (c2) at (5,-0.3) {};

\draw (b2) -- (ab2);

\node[circle,fill,inner sep=0,minimum size=2mm,label={above:$1$}] (top) at (0,2.5) {};
\node[circle,fill,inner sep=0,minimum size=2mm,label={below:$0$}] (bottom) at (0,-1.7) {};

\draw (top) -- (a1);
\draw (top) -- (b1);
\draw (bottom) -- (a1);
\draw (bottom) -- (b1);
\draw (bottom) -- (a2) -- (ab2) -- (top) -- (bc2) -- (c2) -- (bottom);
\draw (c2) -- (ac2) -- (top);
\draw (bottom) -- (b2) -- (bc2);
\draw (a2) -- (ac2);

\end{tikzpicture}
\end{center}
\caption{\label{fig:OML} Example of a complete atomistic orthomodular lattice that does not have the covering property.}
\end{figure}

\begin{lemma} \label{lem:image-of-partial-isometry}
Let $X$ and $Y$ be atomistic Dacey spaces and let $f \colon X \to Y$ and $g \colon Y \to X$ be an adjoint pair of maps such that $f \prl f \circ g \circ f$. Then $\image P(f)$ is an orthoclosed subset of $P(Y)$ and $\image P(g)$ is an orthoclosed subset of $P(X)$. In particular, $P(f)$ is a partial orthometry between $P(X)$ and $P(Y)$, and $P(g)$ is a generalised inverse of $P(f)$.
\end{lemma}

\begin{proof}
We will only show that $\image P(f)$ is orthoclosed. Note that we have $g \prl g \circ f \circ g$. Hence it will follow similarly that $P(g)$ is orthoclosed, and the remaining assumptions will follow from criterion (d) in Proposition~\ref{prop:partial-isometry}.

Let $y \in (\image f)\cc$. We have to show that $y$ is equivalent to an element of $\image f$. We may assume $y \neq 0$. By Lemmas~\ref{lem:lattice-adjoint}(ii) and~\ref{lem:adjointness-continuity}(ii),
\begin{equation} \label{fml:image-of-partial-isometry}
\begin{split} (\image f)\cc
& \;=\; f \big(\{ g(y) \}\c \vee \{ g(y) \}\cc\big)\cc \\
& \;=\; f\big(\{ g(y) \}\c\big)\cc \;\vee\; f\big(\{ g(y) \}\cc\big)\cc \\
& \;=\; f\big(\{ g(y) \}\c\big)\cc \;\vee\; \{f(g(y))\}\cc. \end{split}
\end{equation}
From $\{g(y)\}\c \perp g(y)$ it follows $\{g(y)\}\c \perp g(f(g(y)))$ and hence $f(\{g(y)\}\c) \perp f(g(y))$. It also follows $y \perp f(\{ g(y) \}\c)$. Hence, by orthomodularity, we conclude from (\ref{fml:image-of-partial-isometry}) that $\{y\}\cc \subseteq \{f(g(y))\}\cc$. Since $X$ is atomistic, this means $y \prl f(g(y)) \in \image f$.
\end{proof}

If the orthosets dealt with in Lemma~\ref{lem:image-of-partial-isometry} are irredundant, the statement simplifies. We get in this case a convenient characterisation of partial orthometries between Fr\' echet Dacey spaces.

\begin{theorem}
Let $f \colon X \to Y$ and $g \colon Y \to X$ be an adjoint pair of maps between Fr\' echet Dacey spaces.
\begin{itemize}

\item[\rm (i)] $f$ is a partial orthometry if and only if $f = f \circ g \circ f$.

\item[\rm (ii)] $f$ is a orthometry if and only if $g \circ f = \id_X$.

\item[\rm (iii)] $f$ is a coorthometry if and only if $f \circ g = \id_Y$.

\end{itemize}
In each of these cases, $g$ is the generalised inverse of $f$.
\end{theorem}

\begin{proof}
We only show part (i); the remaining parts are seen similarly.

Let $f$ be a partial orthometry. By irredundancy, $g$ is the unique adjoint and hence the generalised inverse of $f$. Therefore $f = f \circ g \circ f$ by Proposition~\ref{prop:partial-isometry}.

Conversely, assume that $f = f \circ g \circ f$. By Lemma~\ref{lem:image-of-partial-isometry}, $f$ is a partial orthometry and $g$ its generalised inverse.
\end{proof}

We likewise get an easy description of projections of Fr\' echet Dacey spaces.

\begin{lemma} \label{lem:projection}
Let $X$ be an atomistic Dacey space and let $p \colon X \to X$. If $p$ is a projection, then $p$ is idempotent and self-adjoint. Conversely, if $p$ is idempotent and self-adjoint, then $p$ is equivalent to a projection.
\end{lemma}

\begin{proof}
The first part holds by Lemma~\ref{lem:projections}.

Assume that $p$ is idempotent and self-adjoint. Then, by Lemma~\ref{lem:image-of-partial-isometry}, $\image P(p)$ is orthoclosed in $P(X)$. Let $A = (\image p)\cc$. We readily check that then $p(a) \prl a$ for any $a \in A$. We put $\tilde p \colon X \to X \komma \footnotesize x \mapsto \begin{cases} x & \text{if $x \in A$}, \\ p(x) & \text{otherwise.} \end{cases}$ Then $\tilde p \prl p$, and $\tilde p$ is still idempotent and self-adjoint. Moreover, $\image \tilde p = A$ is orthoclosed. Hence $\tilde p$ is a projection by Lemma~\ref{lem:projections}.
\end{proof}

\begin{theorem} \label{thm:projection}
Let $X$ be a Fr\' echet Dacey space and let $p \colon X \to X$. Then $p$ is a projection if and only if $p$ is idempotent and self-adjoint.
\end{theorem}

\begin{proof}
This is clear from Lemma \ref{lem:projection}.
\end{proof}

\section{A category of orthosets}
\label{sec:OS}

A major issue underlying the present work is the problem of how to organise orthosets into a category. It is not an easy matter to decide what kind of maps should be chosen as the morphisms. As to be expected, we use the adjointability condition in this paper.

Let $\OS$ be the category whose objects are all orthosets and whose morphisms are all adjointable maps between orthosets. This definition makes sense, for, as we noticed at the beginning of Section~\ref{sec:Adjointable-maps}, the identity map on an orthoset is adjointable, and the composition of two adjointable maps is again adjointable.

Alternatively, we could choose as morphisms maps that preserve the orthogonality relation. We exploited this approach in our previous work \cite{PaVe1,PaVe2}. We note, however, that adjointability depends likewise on the orthogonality relation in a straightforward way. Moreover, this concept offers a greater flexibility and is better suited for a categorical approach to inner-product spaces. 

We shall see in this section that monomorphisms and epimorphisms in $\OS$ are the injective and surjective maps, respectively. By means of a counterexample, we moreover observe that $\OS$ does not possess equalisers.

\begin{remark}
Let $F$ be one of \/ $\Reals$ or $\Complexes$ and let $\Hil{F}$ be the category of Hilbert spaces over $F$, seen as orthosets, and bounded linear maps between them. By Example~\ref{ex:Adjoints-in-Hilbert-spaces}, $\Hil{F}$ is a subcategory of $\OS$.
\end{remark}

A {\em zero object} of a category is an object $0$ that is both initial and terminal. This means that there are, for any $A$, unique morphisms $0 \to A$ and $A \to 0$. For objects $A$ and $B$, $0_{A,B}$ denotes in this case the morphism factoring through $0$, called the {\it zero map} from $A$ to $B$.

We denote by $\Zero$ the orthoset consisting solely of falsity, called the {\it zero orthoset}. In addition, for use in several proofs that follow, we let $\One$ be an orthoset that contains a single proper element $p$. 

\begin{lemma} \label{lem:nI}
Let $X$ and $Y$ be orthosets.
\begin{itemize}

\item[\rm (i)] $\Zero$ is the zero object of $\OS$. The morphism $0_{\Zero,X} \colon \Zero \to X$ is the map sending $0$ to $0$.

\item[\rm (ii)] The zero map $0_{X,Y}\colon X \to Y$ has the unique adjoint $0_{Y,X}\colon Y\to X$.

\item[\rm (iii)] Every map $f \colon \One \to X$ such that $f(0) = 0$ possesses a unique adjoint. 

\end{itemize}
\end{lemma}

\begin{proof}
Ad (i): Let $0_{\Zero,X} \colon \Zero \to X$ be the map such that $0_{\Zero,X}(0) = 0$, and let $0_{X,\Zero}$ be the map from $X$ to $\Zero$. We have $0_{\Zero,X}(0) \perp x$ and $0 \perp 0_{X,\Zero}(x)$ for any $x \in X$, hence $0_{\Zero,X}$ and $0_{X,\Zero}$ are an adjoint pair. By Lemma~\ref{lem:adjointness-continuity}(i), $0_{\Zero,X}$ is the unique adjointable map from $\Zero$ to $X$, and $0_{X,\Zero}$ is the unique map from $X$ to $\Zero$. The assertions follow.

Ad (ii): By part (i), $0_{Y,X} = 0_{\Zero,X} \circ 0_{Y,\Zero}$ is an adjoint of $0_{X,Y} = 0_{\Zero,Y} \circ 0_{X,\Zero}$. Moreover, the only map equivalent to $0_{Y,X}$ is $0_{Y,X}$ itself, which shows the uniqueness assertion.

Ad (iii): We have $\One = \{p,0\}$. Let $f \colon \One \to X$ be such that $f(0) = 0$. Then a map $g \colon X \to \One$ is an adjoint of $f$ if and only if, for any $x \in X$, we have that $f(p) \perp x$ iff $p \perp g(x)$. Hence $f$ has the unique adjoint $g \colon X \to \One \komma x \mapsto \footnotesize \begin{cases} 0 & \text{if $x \perp f(p)$,} \\ p & \text{otherwise.} \end{cases}$
\end{proof}

We shall characterise the monomorphisms and epimorphisms in $\OS$. We will apply the so-called doubling point construction, explained in the following lemma; cf.\ also \cite[Lemma 4.1]{PaVe2}.

\begin{lemma} \label{lem:doubling}
Let $(X,\perpe{X})$ be an orthoset and let $u \in X\withoutzero$. Let $Z$ arise from $X$ by replacing $u$ with two new elements $u_1$ and $u_2$, and endow $Z$ with the relation $\perpe{Z}$ as follows: for $x, y \in Z \setminus \{u_1,u_2\}$ such that $x \perpe{X} y$, let $x \perpe{Z} y$, and for $x \in Z \setminus \{u_1,u_2\}$ such that $x \perpe{X} u$, let $u_1, u_2 \perpe{Z} x$ and $x \perpe{Z} u_1, u_2$. Then $(Z, \perpe{Z})$ is an orthoset.

Moreover, let $h_1, h_2 \colon X \to Z$ be given as follows: $h_1(x) = h_2(x) = x$ if $x \neq u$; $h_1(u) = u_1$; and $h_2(u) = u_2$. Then $h_1, h_2$ are adjointable.
\end{lemma}

\begin{proof}
Evidently, $Z$ is an orthoset. Note that the difference between $X$ and $Z$ is that the element $u$ of $\class u$ is replaced with two new elements $u_1$ and $u_2$. Thus the map $k \colon Z \to X$ defined by $k(x) = x$ if $x \neq u_1,u_2$ and $k(u_1) = k(u_2) = u$, is an adjoint of both $h_1$ and $h_2$.
\end{proof}

\begin{proposition} \label{lem:monoepi}
Let $f \colon X \to Y$ be a morphism  in $\OS$. Then we have:
\begin{enumerate}[{\rm(i)}]
\item $f$ is a monomorphism in $\OS$\ if and only if $f$ is injective.
\item $f$ is an epimorphism in $\OS$\ if and only if $f$ is surjective.
\end{enumerate}
\end{proposition}

\begin{proof}
Ad (i): To show the ``only if'' part, assume that $f$ is a monomorphism in $\OS$. Let $x_1, x_2 \in X$ be such that $f(x_1) = f(x_2)$. By Lemma~\ref{lem:nI}(iii), there are morphisms $\widehat{x_1}, \widehat{x_2}\colon \One \to X$ such that $\widehat{x_1}(p) = x_1$ and $\widehat{x_2}(p) = x_2$. Then $f \circ \widehat{x_1} = f \circ \widehat{x_2}$ and hence $\widehat{x_1} = \widehat{x_2}$. We conclude $x_1 = x_2$, that is, $f$ is injective. The ``if'' part is evident.
 
Ad (ii): Assume that $f$ is an epimorphism and that there is a $u \in Y \setminus \image f$. Let $Z = (Y \setminus \{u\}) \cup \{u_1,u_2\}$, where $u_1,u_2$ are new elements, be the orthoset as explained in Lemma~\ref{lem:doubling}, and let $h_1, h_2 \colon Y \to Z$ be the morphisms such that $h_1(y) = h_2(y) = y$ if $y \neq u$, $h_1(u) = u_1$, and $h_2(u) = u_2$. Then $h_1 \circ f = h_2 \circ f$ implies $h_1 = h_2$, a contradiction. This shows the ``only if'' part, and again, the ``if'' part is obvious.
\end{proof}

The next proposition shows that equalisers of certain pairs of morphisms exist in $\OS$. 

\begin{proposition}\label{prop:equal}
Let $f,g \colon X \to Y$  be morphisms  in $\OS$\ such that
\[ X_{f,g} \;=\; \{ x \in X \colon f(x) = g(x)\} \]
is a subspace of $X$, and there is a Sasaki map from $X$ to $X_{f,g}$. Then
\[ \begin{tikzcd}
X_{f,g} \arrow[r,  "\iota"]
& X \arrow[r, yshift=0.7ex,  "f"]
\arrow[r, yshift=-0.7ex,  "g"']
& Y
\end{tikzcd} \]
is an equaliser of the pair $f$, $g$, where $\iota \colon X_{f,g} \to X$ is the inclusion map.
\end{proposition}
\begin{proof}
Evidently, $f\circ \iota = g\circ \iota$. Let $h \colon Z \to X$ be a morphism in $\OS$ such that $f \circ h = g \circ h$. Then $\image h\subseteq X_{f,g}$. Put $\overline{h} = \sigma \circ h \colon Z \to X_{f,g}$, where $\sigma \colon X \to X_{f,g}$ is a Sasaki map. Then we have the following commutative diagram
\[ \begin{tikzcd}
Z\arrow[dr,  "h"]\arrow[swap, d,  "\overline{h}"]& &\\
X_{f,g} \arrow[r,  "\iota"]
& X \arrow[r, yshift=0.7ex,  "f"]
\arrow[r, yshift=-0.7ex,  "g"']
& Y,
\end{tikzcd} \]
as $\iota\circ \overline{h} = \iota \circ \sigma \circ h = h$. Moreover, for any further morphism $k \colon Z \to X_{f,g}$ in $\OS$ such that $\iota \circ k = h$, we have $k = \sigma \circ \iota \circ k = \sigma \circ h = \overline{h}$. 
\end{proof}

To show that Proposition~\ref{prop:equal} cannot be generalised to arbitrary pairs of maps, we use the following example.

\begin{example} \label{ex:Example-square}
Consider the following orthoset $X$; cf. \cite[Example 2.15]{PaVe1}:
\begin{center}
\begin{tikzpicture}[
 every node/.style={draw, circle, minimum size=0.675435cm, inner sep=1pt},
 line/.style={draw}
]

\node (s) at (0,2) {$s$};
\node (t) at (2,2) {$t$};
\node (u) at (4,2) {$u$};
\node (v) at (4,0) {$v$};
\node (w) at (4,-2) {$w$};
\node (x) at (2,-2) {$x$};
\node (y) at (0,-2) {$y$};
\node (z) at (0,0) {$z$};
\node (zero) at (2,0) {$0$};

\draw[line] (s) -- (t);
\draw[line] (t) -- (u);
\draw[line] (u) -- (v);
\draw[line] (v) -- (w);
\draw[line] (w) -- (x);
\draw[line] (x) -- (y);
\draw[line] (y) -- (z);
\draw[line] (z) -- (s);

\draw[line] (z) to[bend right] (zero);
\draw[line] (zero) to[bend left=60] (t);
\draw[line] (zero) to[bend right=40] (v);
\draw[line] (zero) to[bend left] (s);
\draw[line] (zero) to[bend left=10] (u);
\draw[line] (zero) to[bend left] (y);
\draw[line] (zero) to[bend right=10] (w);
\draw[line] (x) to[bend right=10] (zero);
\draw[line] (y) -- (x);
\draw[line] (x) -- (w);

\draw[line] (zero) to[out=0,in=-30,looseness=8] (zero);
\end{tikzpicture}
\end{center}
Here, two elements are orthogonal if they either lie both on a straight line or they are connected by a curved line. For instance, $s$, $t$, and $u$ are mutually orthogonal.

$X$ is not a Dacey space. Indeed, $\{s,0\}$ and $\{s,w,0\}$ are subspaces but there is no subspace $A$ orthogonal to $\{s,0\}$ such that $\{s,0\} \vee A = \{s,w,0\}$.
\end{example}

\begin{proposition} \label{prop:NOS-equalisers}
The category $\OS$ does not have equalisers.
\end{proposition}

\begin{proof}
Let $X$ be the orthoset from Example~\ref{ex:Example-square} and let $f \colon X \to X \komma 0\mapsto 0, s \mapsto s, \, t \mapsto z, \, u \mapsto y, \, v \mapsto x, \, w \mapsto w, \, x \mapsto v, \, y \mapsto u, \, z \mapsto t$. Then $f$ is an orthoautomorphism of $X$ and hence, by Proposition~\ref{prop:unitary-maps}, a morphism of $\OS$.

Assume, for sake of contradiction, that the pair of arrows
\begin{tikzcd}
X \arrow[r, yshift=0.7ex,  "f"]
\arrow[r, yshift=-0.7ex,  "\id_X"']
& X
\end{tikzcd}
in $\OS$ possesses the equaliser $e \colon Y \to X$. Then $f \circ e = e$ implies that $\image e \subseteq \{s, w, 0\}$. We claim that actually $\image e = \{s, w, 0\}$. Indeed, assume that $s \not\in \image e$. By Lemma~\ref{lem:nI}(iii) there is a unique morphism $\widehat{s} \colon \One \to X \komma p \mapsto s$. Then $f \circ \widehat{s} = \id_X \circ \widehat{s}$, but there is no adjointable map $k \colon \One \to Y$ such that $ \widehat{s} = e \circ k$. Hence $s \in \image e$, and we argue similarly to see that also $w \in \image e$.

Let now $\tilde e$ be an adjoint of $e$. We claim that $e(\tilde e(s)) = s$. Indeed, otherwise $e(\tilde e(s)) \perp v$ and hence $s \perp e(\tilde e(v))$. But then $e(\tilde e(v)) = 0$ and hence $\tilde e(v) = 0$, that is, $v \in \kernel \tilde e = (\image e)\c = \{u,y,0\}$, and the claim follows.

We may similarly argue to conclude that also $e(\tilde e(s)) = w$ holds. Consequently, the pair $f$, $\id_X$ does not possess an equaliser.
\end{proof}

\section{A dagger category of irredundant orthosets}
\label{sec:iOS}

In our final section, we adapt our categorical framework to the case when the orthosets under consideration are irredundant. The morphisms are again assumed to be adjointable maps. Recall that, by Lemma~\ref{lem:irredundancy-adjoints-1}, adjoints of maps between irredundant orthosets are unique. Therefore it now makes sense to equip the category with an additional structure: a functor that assigns to each morphism its unique adjoint.

A {\it dagger} on a category $\C$ is an involutive functor $\adj \colon \C\op \to \C$ that is the identity on objects. More explicitly, the requirements are that
\begin{align*}
& (g \circ f)\adj \;=\; f\adj \circ g\adj \quad\text{for any morphisms $f, g$,} \\
& {\id_A}\adj \;=\; \id_A \quad\text{for any object $A$,} \\
& f\adjadj \;=\; f \quad\text{for any morphism $f$.}
\end{align*}
A category equipped with a dagger is called a {\it dagger category}. We note that this concept has occurred since the 1960s in the literature and was typically considered in some special context. It entered the mainstream discussion about the foundations of quantum mechanics with Abramsky and Coecke's paper \cite{AbCo}. The notion ``dagger category'' was coined by P.\ Selinger \cite{Sel}. The present work presumes that the reader is acquainted with the fundamental concepts and results on dagger categories. For an introduction to this topic, we may refer, e.g., to \cite{HeJa,Jac} or to Heunen and Vicary's monograph \cite{HeVi}.

Let $\iOS$ be the category consisting of all irredundant orthosets and all adjointable maps between them. The unique adjoint of a morphism $f \colon X \to Y$ is denoted by $f\adj \colon Y \to X$. Equipped with $\adj$, $\iOS$ obviously becomes a dagger category.

In the following, we will see that $\iOS$ shares some properties with the category $\OS$ of all orthosets. In particular, we will consider monomorphisms, epimorphisms, equalisers, and coequalisers in $\iOS$. We will show that, again, $\iOS$ has no equalisers or coequalisers in general. We will also use our categorical framework to discuss the relationship between orthosets and ortholattices. The transition from irredundant orthosets to complete ortholattices, which we defined in Section~\ref{sec:Orthosets-with-0}, will take the form of a dagger-preserving functor $\mathsf C \colon \iOS \to \cOL$. We also consider the same assignment between certain subcategories of $\iOS$ and $\cOL$. For instance, $\mathsf C$ establishes a dagger equivalence between the dagger categories of Fr\' echet orthosets and of complete atomistic ortholattices.

\begin{remark}
Let $F$ be the field of real or complex numbers.  Let $\PHil{F}$ be the category whose objects are the orthosets $P(H)$, where $H$ is a Hilbert space over $F$, and whose morphisms are $P(\phi)$, where $\phi$ is a bounded linear map between Hilbert spaces. Again, by Example~\ref{ex:Adjoints-in-Hilbert-spaces}, $\PHil{F}$ is a subcategory of $\iOS$.
\end{remark}

We start our discussion of $\iOS$ by considering monomorphisms and epimorphisms.

\begin{remark}\label{limitscolimits}
In a dagger category, limits are also colimits and conversely: applying $\adj$ to a limit cone yields a colimit cone and vice versa. Similarly, a morphism $f$ in a  dagger category is a monomorphism if and only if $f\adj$  is an epimorphism.
\end{remark}

\begin{proposition} \label{prop:iosmonoepi}
Let $f \colon X \to Y$ be a morphism  in $\iOS$. 
Then we have:
\begin{enumerate}[{\rm(i)}]
\item $f$ is a monomorphism in $\iOS$\ if and only if $f$ is injective.
\item $f$ is an epimorphism in $\iOS$\ if and only if $f\adj$ is injective.
\end{enumerate}
\end{proposition}

\begin{proof}
Ad (i): This can be seen similarly to Proposition~\ref{lem:monoepi}(i). Note that $\One$ is an irredundant orthoset.

Ad (ii): This follows from part (i) and Remark \ref{limitscolimits}.
\end{proof}

\begin{proposition} \label{prop:iosequal}
Let $f,g\colon X\to Y$  be morphisms  in $\iOS$\ such that
\[ X_{f,g} \;=\; \{ x \in X \colon f(x) = g(x)\} \]
is an irredundant subspace of $X$, and there is a Sasaki map $\sigma \colon X \to X_{f,g}$. Then the following hold:
\begin{itemize}

\item[\rm (i)]
$\begin{tikzcd}
X_{f,g} \arrow[r,  "\iota"]
& X \arrow[r, yshift=0.7ex,  "f"]
\arrow[r, yshift=-0.7ex,  "g"']
& Y
\end{tikzcd}$
is an equaliser in $\iOS$\ of the pair $f$, $g$, where $\iota \colon X_{f,g} \to X$ is the inclusion map.

\item[\rm (ii)]
$\begin{tikzcd}
Y \arrow[r, yshift=0.7ex,  "f\adj"]
\arrow[r, yshift=-0.7ex,  "g\adj"']
& X \arrow[r,  "\sigma"]&X_{f,g}
\end{tikzcd}$
is a coequaliser in $\iOS$\ of the pair $f\adj$, $g\adj$.

\end{itemize}
\end{proposition}

\begin{proof}
Ad (i): As $\iOS$, seen as an ordinary category, is a full subcategory of $\OS$, the assertion follows from Proposition~\ref{prop:equal}.

Ad (ii): This follows from part (i) and Remark \ref{limitscolimits}.
\end{proof}

\begin{proposition} \label{prop:iNOS-equalisers}
The category $\iOS$ does not have equalisers or coequalisers.
\end{proposition}

\begin{proof}
We note that the orthoset $X$ from Example~\ref{ex:Example-square} is irredundant. To see that $\iOS$ does not have equalisers, we may hence argue as in the proof of Proposition~\ref{prop:NOS-equalisers}. The non-existence of coequalisers then follows from Remark~\ref{limitscolimits}.
\end{proof}

With any orthoset $X$, we may associate the ortholattice ${\mathsf C}(X)$. Moreover, a map $f \colon X \to Y$ between orthosets gives rise to the map ${\mathsf C}(f) \colon {\mathsf C}(X) \to {\mathsf C}(Y)$ between the associated ortholattices. We shall now see that these assignments actually define a functor.

Recall that any ortholattice gives rise to the irredundant orthoset $L\supOS$ and, by Remark~\ref{rem:orthosets-and-ortholattices}, $L$ can be recovered from $L\supOS$. We denote by $\cOL$ be the dagger category of all complete ortholattices. Regarding the objects of $\cOL$ as orthosets, we require $\cOL$ to be a full dagger subcategory of $\iOS$. That is, the morphisms of $\cOL$ are the adjointable maps and the dagger is the unique adjoint.

A morphism $f \colon A \to B$ of a dagger category is called a {\it dagger isomorphism} if $f\adj \circ f = \id_A$ and $f \circ f\adj = \id_B$.

\begin{lemma} \label{lem:isos-between-ortholattices}
Let $h \colon L \to M$ be a map between complete ortholattices. Then the following are pairwise equivalent:
\begin{itemize}

\item[\rm (a)] $h$ is an isomorphism of ortholattices.

\item[\rm (b)] $h$, seen as a map between orthosets, is an orthoisomorphism.

\item[\rm (c)] $h$, seen as a map between orthosets, is a dagger isomorphism in $\iOS$.

\end{itemize}
\end{lemma}

\begin{proof}
(a) $\Rightarrow$ (b): Let $h \colon L \to M$ be an isomorphism of ortholattices. Then $h$ is clearly an orthoisomorphism between $L\supOS$ and $M\supOS$.

(b) $\Rightarrow$ (a): Let $h \colon L\supOS \to M\supOS$ be an orthoisomorphism. Then $h$ is adjointable by Proposition~\ref{prop:unitary-maps} and by Lemma~\ref{lem:adjointable-maps-between-ortholattices} order-preserving. It follows that $h$ is an isomorphism of ortholattices.

(b) $\Leftrightarrow$ (c): This equivalence holds by Proposition~\ref{prop:unitary-maps}.
\end{proof}

\begin{theorem} \label{thm:iOS-to-cOL}
$\mathsf C$ is a faithful and dagger essentially surjective dagger-preserving functor from $\iOS$ to $\cOL$.
\end{theorem}

\begin{proof}
Note first that $\mathsf C$ maps $\iOS$ indeed to $\cOL$. For, ${\mathsf C}(X)$ is a complete ortholattice for each $X \in \iOS$. Moreover, if $f \colon X \to Y$ is adjointable, then so is ${\mathsf C}(f) \colon {\mathsf C}(X) \to {\mathsf C}(Y)$ by Lemma~\ref{lem:lattice-adjoint}(i).

Clearly, ${\mathsf C}(\id_X) = \id_{{\mathsf C}(X)}$ for any orthoset $X$ and by Lemma~\ref{lem:adjointness-continuity}(ii), ${\mathsf C}(g \circ f) = {\mathsf C}(g) \circ {\mathsf C}(f)$ for any morphisms $f$ and $g$ of $\OS$. Moreover, ${\mathsf C}(f)\adj = {\mathsf C}(f\adj)$ by Lemma~\ref{lem:lattice-adjoint}(i).

A complete ortholattice $L$ is by Remark~\ref{rem:orthosets-and-ortholattices} isomorphic with ${\mathsf C}(L\supOS)$. By Lemma~\ref{lem:isos-between-ortholattices}, any isomorphism between ortholattices is a dagger isomorphism of $\cOL$. We conclude that $\mathsf C$ is dagger essentially surjective.

Finally, let $f, g \colon X \to Y$ be morphisms of $\iOS$ such that ${\mathsf C}(f) = {\mathsf C}(g)$. For any $x \in X$, we then have that $\{f(x)\}\cc = f(\{x\}\cc)\cc = {\mathsf C}(f)(\{x\}\cc) = {\mathsf C}(g)(\{x\}\cc) = g(\{x\}\cc)\cc = \{g(x)\}\cc$ by Lemma~\ref{lem:adjointness-continuity}(ii). As $Y$ is irredundant, it follows $f = g$. We conclude that $\mathsf C$ is faithful.
\end{proof}

We note that the functor $\mathsf C$ in Theorem \ref{thm:iOS-to-cOL} is not the adjoint of the embedding functor from $\cOL$ to $\iOS$.

\begin{proposition} \label{prop:iOS-to-cOL-adjoints}
The functor $\mathsf C \colon \iOS \to \cOL$ has neither a left nor a right adjoint.
\end{proposition}

\begin{proof}
Assume that ${\mathsf C}$ has the left adjoint $F \colon \cOL\to \iOS$. Then there is for all objects $L\in \cOL$ and $X\in \iOS$ a bijection between the homsets $\homset_{\iOS}(FL,X)$ and $\homset_{\cOL}(L,{\mathsf C}(X))$.

Let ${\mathsf 2} = \{0,1\}$ be the ortholattice with two elements and recall that $\One$ is the orthoset containing a single proper element $p$. As ${\mathsf C}(\One)$ is isomorphic to $\mathsf 2$, $\homset_{\cOL}({\mathsf 2}, {\mathsf C}(\One))$ contains precisely two elements: the zero map and the unique isomorphism between ${\mathsf C}(\One)$ and ${\mathsf 2}$. Therefore $\homset_{\iOS}(F{\mathsf 2}, \One)$ has two elements and as $\iOS$ is a dagger category, so has $\homset_{\iOS}(\One, F{\mathsf 2})$. It follows from Lemma \ref{lem:nI}(iii) that $F{\mathsf 2}$ has likewise two elements, that is, $F{\mathsf 2}$ is orthoisomorphic to $\One$.

Let now $A$ be a $4$-element Boolean algebra. Then ${\mathsf B}(A)$ is a $3$-element Fr\' echet orthoset and $A$ is isomorphic to ${\mathsf C}({\mathsf B}(A))$. Moreover, $\card \homset_{\iOS}(F{\mathsf 2}, {\mathsf B}(A)) = \card {\mathsf B}(A) = 3$ and $\card \homset_{\cOL}({\mathsf 2}, {\mathsf C}({\mathsf B}(A))) = \card \homset_{\cOL}({\mathsf 2}, A) = \card A = 4$. We conclude that there cannot be a bijection between the homsets $\homset_{\iOS}(F{\mathsf 2}, {\mathsf B}(A))$ and $\homset_{\cOL}({\mathsf 2}, {\mathsf C}({\mathsf B}(A)))$.

Assume now that ${\mathsf C}$ has the right adjoint $G \colon \cOL\to \iOS$. Then there is for all objects $L \in \cOL$ and $X \in \iOS$ a bijection between $\homset_{\iOS}(X, GL)$ and $\homset_{\cOL}({\mathsf C}(X),L)$. As both $\iOS$ and $\cOL$ are dagger categories, $\homset_{\iOS}(GL,X)$ and $\homset_{\cOL}(L, {\mathsf C}(X))$ have equal cardinality for all $X$ and $L$. But as above we see that $G{\mathsf 2}$ is orthoisomorphic to $\One$ and hence $\homset_{\iOS}(G{\mathsf 2}, {\mathsf B}(A))$ and $\homset_{\cOL}({\mathsf 2}, {\mathsf C}({\mathsf B}(A)))$ have distinct cardinalities.
\end{proof}

We may reduce our categories with the effect of achieving fullness of the functor between them. Let $\FOS$ be the dagger category of Fr\' echet orthosets. Moreover, let $\caOL$ be the dagger category whose objects are the complete atomistic ortholattices and whose morphisms are the adjointable maps $f \colon L \to M$ between them with the additional property that both $f$ and $f\adj$ send basic elements to basic elements.

\begin{theorem} \label{thm:FOS-acOL}
$\mathsf C$ is a dagger-preserving functor from $\FOS$ to $\caOL$. In fact, $\mathsf C$ establishes a dagger equivalence between $\FOS$ and $\caOL$.
\end{theorem}

\begin{proof}
By Proposition~\ref{prop:Frechet-orthosets}, ${\mathsf C}(X) \in \caOL$ for any $X \in \FOS$. Let $f \colon X \to Y$ be a morphism of $\FOS$. The basic elements of ${\mathsf C}(X)$ are the sets $\{x,0\}$, $x \in X$, and similarly for ${\mathsf C}(Y)$. Since ${\mathsf C}(f)(\{x,0\}) = \{f(x),0\}$ for any $x \in X$, it follows that ${\mathsf C}(f)$ sends basic elements to basic elements. In view of Theorem~\ref{thm:iOS-to-cOL}, it is now clear that $\mathsf C$ is a faithful dagger-preserving functor from $\FOS$ to $\caOL$.

Let $L \in \caOL$. By Proposition~\ref{prop:Frechet-orthosets}, ${\mathsf B}(L)$ is a Fr\' echet orthoset and $L$ is isomorphic with ${\mathsf C}({\mathsf B}(L))$. From Lemma~\ref{lem:isos-between-ortholattices} it follows again that $\mathsf C$ is dagger essentially surjective.

It remains to show that $\mathsf C$ is full. The assertion will then follow by \cite[Lemma~5.1]{Vic}. Let $X$ and $Y$ be Fr\' echet orthosets and $h \colon {\mathsf C}(X) \to {\mathsf C}(Y)$ a morphism of $\caOL$. By assumption, there is a map $f \colon X \to Y$ such that $h(\{x,0\}) = \{f(x),0\}$ for any $x \in X$, and a map $g \colon Y \to X$ such that $h\adj(\{y,0\}) = \{g(y),0\}$ for any $y \in Y$. We observe that $f$ is adjointable, having the adjoint $g$. Moreover, $h$ coincides with ${\mathsf C}(f)$ on the set of basic elements. But $h$ is sup-preserving by Lemma~\ref{lem:adjointable-maps-between-ortholattices} and so is ${\mathsf C}(f)$ by Lemma~\ref{lem:lattice-adjoint}(ii). Hence $h = {\mathsf C}(f)$.
\end{proof}

We observe that in this case, $\mathsf C$, understood as in Theorem \ref{thm:FOS-acOL} as a functor from $\FOS$ to $\caOL$, possesses an adjoint. The difference to the situation in Proposition~\ref{prop:iOS-to-cOL-adjoints} is that the morphisms are restricted to those sending basic elements to basic elements.

\begin{proposition} \label{prop:FOS-to-caOL-adjoints}
The functor ${\mathsf C} \colon \FOS \to \caOL$ has a left and a right adjoint.
\end{proposition}

\begin{proof}
This is a consequence of Theorem~\ref{thm:FOS-acOL}.
\end{proof}

We finally consider the effect of the functor $\mathsf C$ on Dacey spaces. We are led to a category that is closely related to the category of orthomodular lattices studied in \cite{Jac,BPL}.

Let $\iDS$ be the full dagger subcategory of $\iOS$ consisting of all irredundant Dacey spaces. Moreover, let $\cOML$ be the dagger category consisting of complete orthomodular lattices and adjointable maps. 

\begin{lemma} \label{lem:inclusion-maps-in-OML-as-OS}
Let $L$ be a complete orthomodular lattice. Then for any subspace $A$ of $L\supOS$, the inclusion map $\iota \colon A \to L\supOS$ is adjointable.
\end{lemma}

\begin{proof}
Let $a \in L$ be such that $A = \principalideal a$. Let $\sigma \colon L \to A \komma x \mapsto (x \vee a\c) \wedge a$. Then we readily check that, for any $x \in A$ and $y \in L$, we have $x \perp y$ if and only if $x \perp \sigma(y)$. We conclude that $\sigma$ is an adjoint of $\iota$.
\end{proof}

\begin{theorem} \label{thm:DS-cOML}
$\mathsf C$ is a faithful and dagger essentially surjective dagger-preserving functor from $\iDS$ to $\cOML$.

If for every subspace $A$ of an orthoset $X \in \iOS$ the inclusion map is an $\iOS$-morphism, then $X$ belongs to $\iDS$. Any complete orthomodular lattice is of the form ${\mathsf C}(X)$ for such an orthoset.
\end{theorem}

\begin{proof}
The first part is clear from Theorem~\ref{thm:iOS-to-cOL} and Remark~\ref{rem:orthosets-and-ortholattices}.

By Theorem~\ref{thm:adjoints-and-Dacey}, an orthoset such that all inclusion maps of its subspaces are adjointable, is a Dacey space. Moreover, by Remark~\ref{rem:orthosets-and-ortholattices}, a complete orthomodular lattice $L$ is isomorphic with ${\mathsf C}(L\supOS)$ and by Lemma~\ref{lem:inclusion-maps-in-OML-as-OS}, the inclusion maps of subspaces of $L\supOS$ are adjointable.
\end{proof}

\subsubsection*{Acknowledgements}

This research was funded in part by the Austrian Science Fund (FWF) 10.55776/ PIN5424624 and the Czech Science Foundation (GA\v CR) 25-20013L.

The authors are indebted to the reviewer for the helpful comments.

\end{document}